\documentclass[a4paper,11pt,oneside,reqno]{amsart}
\usepackage{amscd,amsfonts,amssymb,amsmath,amsthm,bbm}
\usepackage{graphicx,psfrag}
\usepackage[english]{babel}
\usepackage[headinclude,DIV13]{typearea}
\areaset{15.1cm}{25.0cm}
\usepackage[colorlinks=true,
 linkcolor=blue,
 citecolor=blue,
]{hyperref}

\usepackage[T1]{fontenc}
\usepackage{mathrsfs}
\usepackage{amsmath,amsthm,amssymb}
\usepackage{latexsym}
\usepackage{enumerate}
\usepackage{mathrsfs}
\usepackage{stmaryrd}
\usepackage{amsopn}
\usepackage{amsmath}
\usepackage{amssymb}
\usepackage{amsfonts}
\usepackage{soul}
\usepackage{amsbsy}
\usepackage{amscd,indentfirst,epsfig}
\usepackage{amsfonts,amsmath,latexsym,amssymb,verbatim,amsbsy}
\usepackage{amsthm}
\usepackage{colordvi}
\usepackage{pstricks}
\theoremstyle{plain}
\newtheorem{theo}{Theorem}[section]
\newtheorem{lemme}[theo]{Lemma}

\newtheorem{propo}[theo]{Proposition}

\newtheorem{nb}[theo]{Remark}
\newtheorem{defi}[theo]{Definition}
\theoremstyle{definition}

\def \leq {\leqslant}
\def \geq {\geqslant}

\numberwithin{equation}{section}
\def\ind#1{\lower5pt\hbox{$\scriptstyle #1$}}

\def \d {\,\mathrm{d} }
\def \L {\mathbf{L}}

\def \H {\mathcal{H}}

\def \ds {\displaystyle}

\def \D {\mathscr{D}}

\def \ds {\displaystyle}
\def\Q {\mathcal{Q}}

\def\R{{\mathbb R}}
\def\C{{\mathbb C}}
\def \S {{\mathbb S}^2}

\def \u {{u} }
\def \v {{v}}

\def \w {{w}}

\def \fe {F_\alpha}

\def \M {\mathcal{M}}

\def \IS {\int_{\S}}
\def \IR {\int_{\R^3}}
\def \IRR {\int_{\R^3 \times \R^3}}

\def \B {\mathcal{B}}

\def\Re{\operatorname{Re}}
\def\dist{\operatorname{dist}}

\title[Stability of the steady state ]
{{Exponential trend to equilibrium for the inelastic Boltzmann equation
driven by a particle  bath }}

\author{Jos\'{e}  A. Ca\~{n}izo \& Bertrand Lods }

\address{\textbf{Jos\'{e} A.  Ca\~{n}izo},  Departamento de Matem\'{a}tica
  Aplicada, Universidad de Granada, Av. Fuentenueva S/N, 18071
  Granada, Spain.}
\email{canizo@ugr.es}

\address{\textbf{Bertrand Lods}, Dipartimento di Statistica e Matematica Applicata \& Collegio Carlo Alberto, Universit\`{a} degli
Studi di Torino,  Corso Unione Sovietica, 218/bis, 10134 Torino, Italy.}\email{lods@econ.unito.it}

\begin{document}

\maketitle

\begin{abstract}
  We consider the spatially homogeneous Boltzmann equation for
  inelastic hard spheres (with constant restitution coefficient
  $\alpha \in (0,1)$) under the thermalization induced by a host
  medium with a fixed Maxwellian distribution. We prove that the
  solution to the associated initial-value problem converges
  exponentially fast towards the unique equilibrium solution. The
  proof combines a careful spectral analysis of the linearised
  semigroup as well as entropy estimates. The trend towards
  equilibrium holds in the weakly inelastic regime in which $\alpha$
  is close to $1$, and the rate of convergence is explicit and
  depends solely on the spectral gap of the \emph{elastic} linearised
  collision operator.
\end{abstract}

\tableofcontents
\section{Introduction}
\setcounter{equation}{0}

We pursue our investigation initiated in \cite{BiCaLo,BCL} of the
qualitative properties of inelastic hard spheres suspended in a
thermal medium. In a more precise way, we investigate here the large
time behavior of the one-particle distribution function \(f(v,t)\),
\(v\in\R^3\),
\(t>0\)
solution to the following spatially homogeneous Boltzmann equation:
\begin{equation}
\label{be:force} \partial_t f  =  \Q_\alpha(f,f) + \L(f),
\end{equation}
where $\Q_\alpha(f,f)$ is the inelastic quadratic Boltzmann collision
operator, while $\L(f)$ models the forcing term. The parameter
$\alpha$ is the \emph{restitution coefficient}, expressing the degree
of inelasticity of binary collisions between grains:
$0< \alpha \leq 1$, and the purely elastic case is recovered when
$\alpha = 1$.
 
\subsection{Setting of the problem}

As is well documented, dilute granular flows can be described by
kinetic models associated to suitable modifications of the Boltzmann
operator for which hard-sphere collisions are assumed to be inelastic
\cite{Br}: each encounter dissipates a fraction of the kinetic
energy. In absence of energy supply the system cools down and the
corresponding dissipative Boltzmann equation admits only trivial
equilibria. This is no longer the case if the spheres are forced to
interact with an external thermostat, in which case the energy supply
may lead to a non trivial steady state. Different kinds of forcing
term have been considered in the literature \cite{ernstbrito,
  noije,MiMo,MiMo2,MiMo3,GPV**}, and the qualitative properties of the
corresponding steady states have been investigated. In particular, the
following questions have been addressed regarding steady solutions of
kinetic equations of the type \eqref{be:force}: (1) Existence
\cite{MiMo, GPV**}, (2) Uniqueness in some weakly inelastic regime
corresponding to $\alpha$ close to $1$ \cite{MiMo2,MiMo3} and (3)
stability, i.e. convergence of the solution to the associated
Boltzmann equation towards the steady state \cite{MiMo2,MiMo3}. In
particular, for hard spheres subject to diffuse forcing, the
large-time behaviour of the solution to the BE has been completely
characterised in \cite{MiMo3} whereas, for anti-drift forcing (closely
related to self-similar solutions to the freely evolving Boltzmann
equation), asymptotic behaviour has been considered in
\cite{MiMo2}.

In this paper we are concerned with the asymptotic behaviour of a
physical model in which the system of inelastic hard spheres is
immersed in a thermal bath of particles at equilibrium, already
investigated in \cite{BiCaLo,BCL}. In this model the forcing term is
given by a linear scattering operator describing \emph{elastic
  collisions} with the background medium:
$$\L (f) = \Q_{1}(f, \M)$$
where $\M$ stands for the distribution function of the host fluid,
supposed to be a Maxwellian with unit mass, bulk
velocity~$u_0 \in \R^{3}$ and temperature $\Theta_0>0$:
\begin{equation}\label{maxwe1}
\M(\v)=\bigg(\dfrac{1}{2\pi
\Theta_0 }\bigg)^{3/2}\exp
\left\{-\dfrac{(\v-\u_0 )^2}{2\Theta_0 }\right\}, \qquad \qquad \v
\in \R^3.
\end{equation}
The precise definitions of the collision operators $\Q_\alpha(f,f)$
and $\L (f)$ are given in Subsection~2.1. We refer to
\cite{MiMo,MiMo2,MiMo3,GPV**} for a mathematical discussion of various
models and their physical motivation. We restrict our attention to
interacting hard spheres and refer to \cite{CarrTo} for exhaustive
references on the pseudo-Maxwell approximation. A salient feature of
both collision operators $\Q_{\alpha}(f,f)$ and $\L$ is that their
only collision invariant is mass, i.e.
$$\IR \Q_{\alpha}(f,f)\d v=\IR \L f\d v=0.$$
In contrast with the elastic Boltzmann operator, neither the momentum
nor the energy are conserved by $\Q_{\alpha}$ or $\L$.

The existence of smooth stationary solutions for the inelastic
Boltzmann equation under the above thermalization has already been
proved in~\cite{BiCaLo}, for any choice of the restitution
coefficient~$\alpha$. This has been achieved by controlling the
$L^p$-norms, the moments and the regularity of the solutions for the
Cauchy problem, together with a dynamical argument based on the
Tychonoff fixed-point theorem.

Uniqueness of the steady state is proven in some previous contribution
\cite{BCL} for a smaller range of parameters $\alpha$. Namely, the
main results of both \cite{BiCaLo} and \cite{BCL} can be summarized as
follows:
\begin{theo}[Existence and Uniqueness of the steady
  state]\label{theo:ex-uni}
  For any $\varrho \geq 0$ and $\alpha \in (0,1]$, there exists a
  steady solution $F_{\alpha} \in L^{1}_{2}(\R^{3})$,
  $F_{\alpha}(v) \geq 0$ to the problem
  \begin{equation}\label{equi}
    \Q_{\alpha}(F_{\alpha},F_{\alpha}) +\L(F_{\alpha})=0
  \end{equation}
  with $\ds\IR F_{\alpha}(v)\d v=\varrho.$ 

  Moreover, there exists $\alpha_0 \in (0,1]$ such that such a
  solution is unique for $\alpha \in (\alpha_0,1]$. The (unique)
  steady state $F_\alpha$ for $\alpha \in (\alpha_0,1]$ is radially
  symmetric and belongs to $\mathcal{C}^\infty(\R^3)$.
\end{theo}

We assume in the sequel that $\alpha \in (\alpha_{0},1]$ and we are
interested in the large-time behaviour of the solution $f(t,v)$ to
\eqref{be:force} (whose existence and uniqueness are guaranteed by
\cite{BiCaLo}). In particular, if one assumes
\begin{equation}\label{mass}
  \IR f_{0}(v)\d v =1,
\end{equation}
then one expects the solution $f(t,v)$ to converge towards the unique
steady solution with unit mass given in Theorem \ref{theo:ex-uni}. Our
goal in the present paper is to provide sufficient conditions on the
initial datum $f_{0}$ ensuring that this convergence holds true, with
an exponential rate that we make explicit. As in \cite{MiMo2,MiMo3},
this exponential trend to equilibrium is proven to hold in the weakly
inelastic regime, i.e. for a range of parameters
$\alpha \in (\alpha^{\dagger},1]$ for a certain explicit
$\alpha^{\dagger} > \alpha_{0}.$

\subsection{Main result and strategy of proof}

Our goal is to prove a quantitative version of the return to
equilibrium for the solution to \eqref{be:force}. Our main result
combines local stability estimates in a certain weighted $L^{1}$-space
with suitable entropy estimates. A crucial point in our approach is
that it strongly relies on the understanding of the elastic problem
corresponding to $\alpha=1.$ In the elastic case, as is well known,
$F_{1}=\M$ is exactly the host-medium Maxwellian appearing in
$\L$. The spectral properties of the linearised operator around $\M$
have been studied in \cite{BCL}. Namely, in the weighted space
$$\mathcal{X}=L^{1}(\R^{3},\exp(a|v|)\d v), \qquad a > 0,$$
the elastic linearised operator $\mathscr{L}_{1}$ given by
\begin{equation*}
  \mathscr{L}_{1}h
  = \Q_{1}(h,\M)+\Q_{1}(\M,h)+\L h,
  \qquad h \in L^{1}\big( \R^{3};\,(1+|v|)\exp(a|v|)\d v \big)
\end{equation*}
admits a positive spectral gap $\nu > 0$ which can be explicitly
estimated. We shall consider solutions to \eqref{be:force} associated
to a nonnegative initial datum $f_{0} \in \mathcal{X}$ satisfying
\eqref{mass} and
\begin{equation}
  \label{eq:Hf0M}
  H(f_{0}| \M) =
  \IR f_{0}(v)\log\left(\dfrac{f_{0}(v)}{ \M(v)}\right)
  \d v < \infty.
\end{equation}
Our main result can be stated as follows:
\begin{theo}\label{theo:main}
  For any $0 < \nu_{*} < \nu$ (with $\nu$ equal to the size of the
  spectral gap of $\mathscr{L}_1$) there exists some explicit
  $\alpha^{\dagger} \in (0,1)$ such that, for any
  $\alpha \in (\alpha^{\dagger},1]$ and any nonnegative initial datum
  with $f_{0} \in \mathcal{X}$ satisfying \eqref{mass} and
  \eqref{eq:Hf0M} the solution $f(t,v)$ to \eqref{be:force} satisfies
$$\|f(t)-\fe\|_{\mathcal{X}}\leq K\exp\left(-\nu_{*} t\right) \qquad \forall t \geq 0$$
for some positive constant $K$  depending on $\varepsilon, \alpha$ and $H(f_{0}|\M)$. 
\end{theo} 

Notice that the rate of convergence is explicitly computable in terms
of the spectral gap of the \emph{elastic} linearised operator
$\mathscr{L}_{1}$ in $\mathcal{X}.$

The proof of the above is based upon two main ingredients:
\begin{enumerate}
\item \emph{A local stability estimate} in which exponential
  convergence is established for small perturbations of the equilibrium
  state, i.e. whenever the initial datum $f_{0}$ is close enough to
  $ \fe.$
\item Suitable \emph{entropy estimates} as a tool to pass from local
  to global stability. We take advantage of the fact that the
  scattering operator $\L$ is dominant in the weakly inelastic
  regime.
\end{enumerate}

To tackle the above first point (1), we have to perform a fine study
of the spectral properties of both the linearised operator
$\mathscr{L}_{\alpha}$ around the steady solution $\fe$ and its
associated evolution semigroup. More precisely, introduce
\begin{equation*}
\mathscr{L}_{\alpha}h=\Q_{\alpha}(h,\fe)+\Q_{\alpha}(\fe,h)+\L h,
\qquad h \in L^{1}(\R^{3}\,;\,(1+|v|)\exp(a|v|)\d v), \alpha \in
(\alpha_{0},1].
\end{equation*}
We deduce the spectral properties of $\mathscr{L}_{\alpha}$ in
$\mathcal{X}$ from those of the elastic operator $\mathscr{L}_{1}$ by
a perturbation argument valid for $\alpha$ close enough to $1$. Notice
that the elastic limit $\alpha \to 1$ is actually well behaved since
the operator gap (in the sense of \cite{kato}; see Appendix
\ref{sec:perturbation}) between $\mathscr{L}_{\alpha}$ and
$\mathscr{L}_{1}$ is going to $0$ as $\alpha \to 1$. This allows us to
apply results from the perturbation theory of unbounded operators
\cite{kato}. This strongly contrasts with the analysis of
\cite{MiMo2,MiMo3} which, though perturbative, was ill-behaved in the
elastic limit.

As is well known, the spectral properties of the $C_{0}$-semigroup
$(\mathcal{S}_{\alpha}(t))_{t \geq 0}$ generated by
$\mathscr{L}_{\alpha}$ cannot be directly deduced from those of
$\mathscr{L}_{\alpha}$ because of the lack of spectral mapping theorem
in infinite dimensional Banach spaces. In particular, one cannot
directly derive from the existence of a spectral gap for
$\mathscr{L}_{\alpha}$ the decay of the associated semigroup. However,
following an operator splitting strategy introduced in \cite{MS}, we
can localise the essential spectrum of
$(\mathcal{S}_{\alpha}(t))_{t\geq 0}$ through a weak compactness
argument and deduce from that the local stability theorem (see Theorem
\ref{theo:local} and Theorem \ref{thm:exp-convergence}). We give a
direct and elementary proof that does not rely on the recent results
of \cite{GMM,MS}.\medskip

In order to address the above point (2) entropy estimates play a
crucial role. Our method is based upon the following entropy-entropy
production estimate recently obtained in \cite{BCL1}. Introducing the
entropy production associated to $\L$,
\begin{equation*}
  \mathbf{D}(f)=-\IR \L f(v)\log\left(\dfrac{f(v)}{\M(v)}\right)\d v,
\end{equation*}
the result reads as follows:
\begin{theo}\label{theo:bcl1}
 There exists $\lambda > 0$ such that
  \begin{equation*}
    \mathbf{D}(f)
    \geq \lambda H(f|\M)
    := \lambda \,\IR f(v)\log\left(\dfrac{f(v)}{\M(v)}\right)\d v
    \geq 0
  \end{equation*}
  holds for any probability distribution $f \in L^{1}(\R^{3},\d v)$.
\end{theo}

This result has important consequences on the asymptotic behaviour of
the solution to \eqref{be:force} in the elastic case $\alpha=1$.  In
this case the scattering operator $\L$ becomes dominant and forces the
solution $f(t,v)$ to \eqref{be:force} to converge exponentially fast
towards the unique equilibrium state of $\L$, which is the Maxwellian
$\M$ (see \cite{BCL1} for details).  Roughly speaking, in the elastic
limit $\alpha \to 1$, we expect the persistence of this behaviour and
we expect the scattering operator $\L$ to drive the system in some
neighbourhood of $\M$. Since $ \fe \simeq \M$ for $\alpha$ close to 1,
the dynamics is forced to take the solution close to $\M$ as
$t \to \infty.$ To be more precise, using the above Theorem
\ref{theo:bcl1}, one can estimatethe evolution of the relative entropy
along the solutions to \eqref{be:force} (see Proposition
\ref{prop:entropy}) to get
\begin{equation*}
  H(f(t)|\M)
  \leq \exp(-\lambda t)H(f_{0}|\M) + K(1-\alpha)
  \qquad \forall t \geq 0\;;\;\alpha \in (\alpha_{0},1]
\end{equation*}
for some positive constant $K > 0$ independent of $\alpha$.  The above
estimate ensures that, for large time and $\alpha \simeq 1$, the
solution to \eqref{be:force} will become close enough to $\M$ and,
hence to $\fe$ which, combined with the local stability theorem,
yields our main result. It is worth mentioning here that, while the
analysis in \cite{BCL} dealt with a (possibly) inelastic scattering
operator, we restrict ourselves here to \emph{elastic} interactions
between the hard spheres and the host medium due to the inavailability
of Theorem \ref{theo:bcl1} in the inelastic case. Notice however that
all the spectral results of the paper, as well as the local stability
theorem \ref{thm:exp-convergence}, hold true without modification
substituting $\L$ by the inelastic scattering operator
\begin{equation*}
  \L_{e}f=\Q_{e}(f,\M),
\end{equation*}
associated to a general constant restitution coefficient
$e \in (0,1]$.

\subsection{Plan of the paper.}

The organisation of the paper is as follows. In the next section we
define the collision operators $\Q_{\alpha}$ and $\L$, and recall from
\cite{BCL} the main properties of the solution to \eqref{be:force} and
the steady state $\fe$. Section 3 is devoted to a study of the
spectral properties of $\mathscr{L}_{\alpha}$ and
$(\mathcal{S}_{\alpha}(t))_{t\geq 0}$, used later in Section 4 to
derive the local stability Theorem \ref{thm:exp-convergence}. In
Section 5 we exploit the entropy estimates to establish our main
global stability result. The proof that, for any
$\alpha\in (\alpha_{0},1]$, $\mathscr{L}_{\alpha}$ generates a
$C_{0}$-semigroup in $\mathcal{X}$ is postponed to Appendix A and uses
a weak compactness argument.

\subsection{Notation.}

Given two Banach spaces $X$ and $Y$, we denote by $\mathscr{B}(X,Y)$
the set of linear bounded operators from $X$ to $Y$ and by
$\|\cdot\|_{\mathscr{B}(X,Y)}$ the associated operator norm. If $X=Y$,
we simply denote $\mathscr{B}(X):=\mathscr{B}(X,X)$. We denote then by
$\mathscr{C}(X)$ the set of closed, densely defined linear operators
on $X$ and by $\mathscr{K}(X)$ the set of all compact operators in
$X$. For $A \in \mathscr{C}(X)$, we write $\D(A) \subset X$ for the
domain of $A$, $\mathscr{N}(A)$ for the null space of $A$ and
$\mathrm{Range}(A) \subset X$ for the range of $A$. The spectrum of
$A$ is then denoted by $\mathfrak{S}(A)$ and the resolvent set is
$\rho(A)$. For $\lambda \in \rho(A)$, $R(\lambda,A)$ denotes the
resolvent of $A$.  We also define the discrete spectrum
$\mathfrak{S}_{\mathrm{d}}(A)$ as the set of eigenvalues of $A$ with
finite algebraic multiplicity (see \cite{kato, engel} for more
details). We denote by $s(A)$ the \emph{spectral bound} of $A$, i.e.
$$s(A)=\sup\{\Re\lambda\,;\,\lambda \in \mathfrak{S}(A)\}.$$
There are several definitions of the essential spectrum of $A$ in the
literature which are unfortunately not equivalent. In the present
paper we adopt the notion of \emph{Schechter essential spectrum},
denoted by $\mathfrak{S}_{\mathrm{ess}}(A)$ and defined by
\begin{equation*}
  \mathfrak{S}_{\mathrm{ess}}(A)=\bigcap_{K \in \mathscr{K}(X)} \mathfrak{S}(A+K).
\end{equation*}
For a bounded operator $T \in \mathscr{B}(X)$ we can also define the
\emph{essential radius} of $T$ as
\begin{equation*}\begin{split}
  r_{\mathrm{ess}}(T) &= \inf\left\{
    r > 0\,;
    \mathfrak{S}(T)\cap \{\lambda \in \mathbb{C}\,;\,|\lambda| > r\}
    \subset \mathfrak{S}_{\mathrm{d}}(T)\,
  \right\}\\
 & =\sup\left\{|\mu|\,;\,\mu \in \mathfrak{S}_{\mathrm{ess}}(T)\right\}.
\end{split}\end{equation*}
Notice that the first identity is peculiar to Schechter essential spectrum whereas the second one is valid for any of the various notions of  essential spectrum (see \cite[Corollary 4.11, p. 44]{edmunds}.
If $(U(t))_{t\geq 0}$ is a $C_{0}$-semigroup in $X$ with generator
$A$, we denote by $\omega_{0}(U)$ its \emph{growth bound} and by
$\omega_{\mathrm{ess}}(U)$ its \emph{essential type}, defined by
\begin{align*}
  &\exp(t\,\omega_{0}(U)) =
  \sup\left\{|\mu|\,;\,\mu \in \mathfrak{S}(U(t))\right\},
  \\
  &\exp(t\,\omega_{\mathrm{ess}}(U))
  = r_{\mathrm{ess}}(U(t))
  = \sup\left\{|\mu|\,;\,\mu \in \mathfrak{S}_{\mathrm{ess}}(U(t))\right\}
  \qquad \text{for $t \geq 0$.}
\end{align*}

\section{Preliminary results}

\subsection{The kinetic model}

Given a constant restitution coefficient $\alpha \in (0,1)$, one
defines the bilinear Boltzman operator $\Q_\alpha$ for inelastic
interactions and hard spheres by its action on test functions
$\psi(v)$:
 \begin{equation}\label{weakfg}
 \IR \Q_\alpha (f,g)(v)\, \psi(v)\d\v  = \IR\IR \IS f(v)g(w)\,|v-w|\,\left(\psi(v')-\psi(v)\right)\d v \d w \d \sigma
 \end{equation}
with $v'=v+\frac{1+\alpha}{4}\,(|v-w|\sigma-v+w)$. In particular, for any test function $\psi=\psi(v)$, one has the following weak form of the \textit{quadratic} collision operator:
 \begin{equation}\label{co:weak}
\begin{split}
\IR \Q_\alpha (f,f)(v)\, \psi(v)\d\v  =  \frac{1}{2} \IR\IR f(v)\,f(w)\,|v-w|
\mathcal{A}_\alpha[\psi](v,w)\d\w\d\v,
\end{split}
\end{equation}
where
\begin{equation}
  \begin{split}
    \label{coll:psi} \mathcal{A}_\alpha[\psi](v,w) &=
    \frac{1}{4\pi}\IS(\psi(v')+\psi(w')-\psi(v)-\psi(w))\d{\sigma}
    \\
    &=\mathcal{A}^+_\alpha[\psi](v,w)-\mathcal{A}^-_\alpha[\psi](v,w)
  \end{split}
\end{equation}
(where we have used the symmetry of the integral under interchange of
$v$ and $w$) and the post-collisional velocities $(v',w')$ are given
by
\begin{equation}
\label{co:transf}
  v'=v+\frac{1+\alpha}{4}\,(|q|\sigma-q),
\qquad  w'=w-\frac{1+\alpha}{4}\,(|q|\sigma-q), \qquad q=v-w.
\end{equation}
In the same way,   one defines the linear scattering operator $\L$ by its action
on test functions:
 \begin{equation}\label{co:weakL}
\begin{split}
\IR \L (f)(v)\, \psi(v)\d\v  =  \frac{1}{2} \IR\IR f(v)\,\M(w)\,|v-w|
\mathcal{J}[\psi](v,w)\d\w\d\v,
\end{split}
\end{equation}
where
\begin{equation}\label{col:Je}
 \mathcal{J}[\psi](v,w)=
\dfrac{1}{4\pi}\int_{\S}
\left(\psi( {v}^\star)-\psi(v)\right)\d\sigma=\mathcal{J}^+[\psi](v,w)-\mathcal{J}^-[\psi](v,w).
\end{equation}
with  post-collisional velocities $(v^\star,w^\star)$
\begin{equation}
\label{co:transf1}
  v^\star=v+\frac{1}{2}\,(|q|\sigma-q),
\qquad  w^\star=w-\frac{1}{2}\,(|q|\sigma-q), \qquad q=v-w.
\end{equation}
For simplicity, we shall assume in the paper that the particles governed by $f$ and those with distribution function $\M$ share the same mass.
Notice that
$$\L(f)=\Q_1(f,\M)$$
and we shall adopt the convention that post (or pre-) collisional
velocities associated to the coefficient $\alpha$ are denoted with a
prime, while those associated to elastic collision are denoted with a
$\star$. We are interested in the large time behaviour of solutions to
the following Boltzmann equation:
\begin{equation}
  \label{BEev}
  \partial_t f(t,v)
  = \Q_\alpha(f(t,\cdot);f(t,\cdot))(v) + \L(f)(t,v),
  \quad f(0,v)=f_0(v),
  \qquad t > 0, v \in \R^3.
\end{equation}
Notice that
$$\Q_\alpha(f,f)=\Q^+_\alpha(f,f)-\Q^-_\alpha(f,f)=\Q^+_\alpha(f,f)-f \mathbf{\Sigma}(f) $$
where
$$\mathbf{\Sigma}(f)(v) = (f \ast |\cdot|)(v) = \IR f(w)|v-w|\d w.$$
Notice that $\mathbf{\Sigma}(f)$ does not depend on the restitution
coefficient $\alpha \in (0,1]$. In the same way,
\begin{equation*}
  \L(f)(v)=\L^+(f)(v)-\L^-(f)(v)
  =
  \L^+(f)(v)-\mathbf{\Sigma}(\M)(v)f(v)
\end{equation*}
Existence and uniqueness of solutions to \eqref{BEev} have been
established in \cite{BiCaLo}. In particular, if $f_{0}$ is a
nonnegative initial datum with
\begin{equation}
  \label{eq:mm}
  \int_{\R^{3}}f_{0}(v)|v|^{3}\d v < \infty
  \qquad \text{ and } \qquad
  \IR f_{0}(v)\d v=1
\end{equation}
then there exists a unique nonnegative solution $(f(t,v))_{t\geq 0}$
to \eqref{BEev} which additionally satisfies
$$\sup_{t\geq 0}\int_{\R^{3}}f(t,v)|v|^{3}\d v
< \infty, \text{ and } \quad \IR f(t,v)\d v=1 \qquad \forall t \geq 0.$$
More generally, uniform propagation of moments holds: namely, for any
$p \geq 2$ one has
\begin{equation}
  \label{eq:propmoments}
  \IR f_{0}(v)|v|^{p}\d v < \infty
  \implies
  \sup_{t \geq 0}\IR f(t,v)|v|^{p}\d v < \infty.
\end{equation}
See \cite[Proposition 4.2 \& Theorem 4.8]{BiCaLo} for more
details. Owing to the above mass conservation property we shall
restrict ourselves in the sequel to nonnegative initial data
satisfying \eqref{eq:mm}. Due to the influence of the scattering
operator $\L$ there is no additional conservation law besides mass
conservation. In fact, it appears impossible to express the evolution
of the momentum
$$\mathbf{u}(t)=\IR f(t,v)v \d v \in \R^{3}$$
and the energy
$$E(t)=\IR f(t,v)|v|^{2}\d v$$
in a closed form.

\subsection{A posteriori estimates}

We collect here several results obtained in our previous contribution
\cite{BCL} regarding the properties of solutions to \eqref{BEev} as
well as those of the steady solution $\fe$ to \eqref{equi}. We begin
with \emph{high-energy tails} for the solution to \eqref{be:force} and
$F_{\alpha}.$
\begin{theo}\label{tailT}
  Let $f_0$ be a nonnegative velocity distribution with
  $\IR f_0(v)\d v=1$. Assume that $f_0$ has an exponential tail of
  order $s \in (0,2]$, i.e. there exists $r_0 >0$ and $s \in (0,2]$
  such that
  \begin{equation*}
    \IR f_0(v)\exp\left(r_0|v|^s\right) \d v <\infty.
  \end{equation*}
  Then there exist $0 < r \leq r_0$ and $C >0$ (independent of
  $\alpha \in (0,1]$) such that the solution $(f(t,v))_{t\geq 0}$ to
  the Boltzmann equation \eqref{BEev} satisfies
  \begin{equation}
    \label{expbounds}
    \sup_{t \geq0}\IR f(t,v)\exp\left(r|v|^s\right) \d v \leq C < \infty.
  \end{equation}
  In particular, there exist constants $A >0$ and $M >0$ such that for
  all $\alpha \in (0,1]$ and all solutions $F_\alpha$ to \eqref{equi}
  one has
  $$\int_{\R^3} F_\alpha(v) \exp\left(A |v|^2\right) \d\v \leq M.$$
\end{theo}
Notice that the above integral tail estimate for $\fe$ can actually be
strengthened to get the following \emph{pointwise} Maxwellian bounds:
\begin{propo}[{\cite[Theorems 4.4 \& 4.7]{BCL}}]
There exist two Maxwellian distributions $\underline{\M}$ and $\overline{\M}$
  (independent of $\alpha$) such that
  $$\underline{\M}(v) \leq \fe(v) \leq \overline{\M}(v)
    \qquad \forall v \in \R^3,
    \qquad \forall \alpha \in (\alpha_0,1).$$
\end{propo}
\subsection{Convergence of $\fe$ to $\M$}\label{sec:convFeM} 

Let us now introduce
\begin{equation}
  \label{eq:XY-def}
  \mathcal{X} =
  L^1(m ^{-1}) = L^1(\R^3,m^{-1}(v)\d v),
  \qquad
  \mathcal{Y} =
  L^1_1(m ^{-1}) = L^1(\R^3, \langle v \rangle m^{-1}(v)\d v)
\end{equation}
where
$$\langle v \rangle=(1+|v|^{2})^{\frac{1}{2}} \quad \text{ and } \quad m(v)=\exp\left(-a |v|\right), \qquad a >0, \:\:v \in \R^{3}.$$
According to the above Theorem \ref{tailT}, $\fe \in \mathcal{X}$ for
any $\alpha \in (0,1]$.  We recall from \cite[Proposition 11]{AlCaGa}
that $\Q_{\alpha}$ is well defined on $\mathcal{Y}$:
\begin{propo}\label{propoAlo} There exists $C >0$ such that, for any $\alpha  \in (0,1)$
$$\|\Q_\alpha(h,g)\|_{\mathcal{X}}+\|\Q_\alpha(g,h)\|_{\mathcal{X}} \leq C\|h\|_{\mathcal{Y}}\,\|g\|_{\mathcal{Y}} \qquad \qquad \forall h,g \in \mathcal{Y}.$$
\end{propo}
Moreover, one has the following:
\begin{propo}[{\cite[Proposition 3.2]{MiMo2}}]
  \label{elasticMM}
  For any $\alpha,\alpha' \in (0,1)$ and any $f \in
  W^{1,1}_1(m^{-1})$, $g \in L^1_1(m^{-1})$, it holds
  \begin{equation*}
    \|\Q_\alpha^+(f,g)-\Q_{\alpha'}^+(f,g)\|_{\mathcal{X}}
    \leq p(\alpha-\alpha')\|f\|_{W^{1,1}_1(m^{-1})}\,\|g\|_{\mathcal{Y}}
  \end{equation*}
  and
  $$\|\Q_\alpha^+(g,f)- \Q_{\alpha'}^+(g,f)\|_{\mathcal{X}} \leq  p(\alpha-\alpha')\|f\|_{W^{1,1}_1(m^{-1})}\,\|g\|_{\mathcal{Y}}$$
  where $ p(r)$ is an explicit polynomial function with
  $\lim_{r \to 0^+} p(r)=0.$
\end{propo}
In the elastic limit $\alpha \to 1$, one has the following:
\begin{theo}[{\cite[Theorem 5.5]{BCL}}]
  \label{limit1}
  There exists an explicit function $\eta_1(\alpha)$ such that
  $\lim_{\alpha \to 1}\eta_1(\alpha)=0$ and such that  \begin{equation*}
    \left\|\fe-\M\right\|_{\mathcal{Y}}
    \leq
    \eta_1(\alpha)
    \qquad \forall \alpha \in (\alpha_0,1].
  \end{equation*}
\end{theo}

\subsection{Spectral properties of the linearised operator for
  $\alpha=1$}
\label{sec:spectr-elast}

Define the elastic linearised operator
$\mathscr{L}_1\::\:\D(\mathscr{L}_1) \subset \mathcal{X} \to
\mathcal{X}$ by
\begin{equation*}
  \mathscr{L}_1 (h)=\Q_1(\M,h)+\Q_1(h,\M) +\L h,
  \qquad \forall h \in \D(\mathscr{L}_1)=\mathcal{Y}.
\end{equation*}
(We recall $\mathcal{X}$ and $\mathcal{Y}$ were defined in
\eqref{eq:XY-def}.)  We introduce also
\begin{equation*}
  \widehat{\mathcal{X}} = \{f \in \mathcal{X}\,;\,\IR f \d v=0\},
  \qquad
  \widehat{\mathcal{Y}}=\{f \in \mathcal{Y}\,;\,\IR f \d v =0\}.
\end{equation*}
One has the following structure of the spectrum of $\mathscr{L}_{1}$:
\begin{theo}[{\cite[Theorem 5.3]{BCL}}]
  \label{spect}
  The null space of $\mathscr{L}_1$ in $\mathcal{X}$ is given by
  \begin{equation*}
    \mathscr{N}(\mathscr{L}_1)=\mathrm{span}(\M).
  \end{equation*}
  Moreover, $\mathscr{L}_1$ admits a positive spectral gap $\nu >0$.
  In particular,
  $\mathscr{N}(\mathscr{L}_1) \cap \widehat{\mathcal{X}}=\{0\}$ and
  $\mathscr{L}_1$ is invertible from $\widehat{\mathcal{Y}}$ to
  $\widehat{\mathcal{X}}$.
\end{theo}

Let us spend a few words on the strategy used to prove the above
result since we will use several of the tools involved to study the
spectral properties of the linearised semigroup in the next section.
The proof of the above result is related to a general strategy
introduced in \cite{GMM} which consists in deducing the spectral
properties in $L^{1}$ from the much easier spectral analysis in
$L^{2}$. The existence of a spectral gap for the linearised collision
operator in $\H=L^{2}(\M^{-1})$ is relatively easy to obtain through a
suitable Poincar\'e-like inequality and the task is to prove that the
linearised collision operator in $\mathcal{X}$ can be deduced from the
one in the Hilbert setting. This is done thanks to a suitable
splitting of the linearised operator as
$$\mathscr{L}_{1}=\mathcal{A}_{1}+ \mathcal{B}_{1}$$
where
\begin{enumerate}[(i)]
\item $\mathcal{A}_{1}\::\:\mathcal{X} \to \H$ is bounded;
\item the operator
  $\mathcal{B}_{1}\::\:\D(\mathcal{B}_{1}) \to \mathcal{X}$ (with
  $\D(\mathcal{B}_{1})=\mathcal{Y}$) is $\beta$-dissipative for some
  positive $\beta >0$, i.e.
  \begin{equation}
    \label{Bdiss}
    \IR \mathrm{sign}f(v) \mathcal{B}_{1}f(v)m^{-1}(v)\d v
    \leq -\beta  \|f\|_{\mathcal{Y}}
    \qquad \forall f \in \mathcal{Y}.
  \end{equation}
\end{enumerate}
Under these conditions \cite{GMM} asserts that the spectrum of
$\mathscr{L}_{1}$ in $\mathcal{X}$ will be the same of that in $\H$.

\medskip

To be more precise, the splitting is as follows. Let us introduce the linearised Boltzmann operator
$$\mathcal{T}_{1}f=\Q_{1}(\M,f)+\Q_{1}(f,\M)$$
so that $\mathscr{L}_{1}=\mathcal{T}_{1}+\L$. Clearly, 
$$\mathcal{T}_{1}f=\mathcal{T}_{1}^{+}f-\sigma_{1}(v)f(v)-\M(v)\,\int_{\R^{3}}\,f(w)|v-w|\d w, \qquad \text{ while }  \qquad\L f(v)=\L^{+}(f)-\mathbf{\Sigma}(v)f(v)$$
and the splitting consists in setting, for some $R > 0$ large enough so that \eqref{Bdiss} holds, 
$$\mathcal{A}_{1}=\mathcal{A}_{1}^{1}+\mathcal{A}_{1}^{2}$$
with 
$$\mathcal{A}_{1}^{1}f=\mathcal{T}^{+}_{1}(\chi_{B_{R}}f)+\L^{+}(\chi_{B_{R}}f), \qquad \mathcal{A}_{1}^{2}f(v)=-\M(v)\,\int_{\R^{3}}\,f(w)|v-w|\d w$$
where $B_{R}$ is the open ball in $R^{3}$ with radius $R >0$ and center $0$.  Then, simply sets $\mathcal{B}_{1}=\mathscr{L}_{1}-\mathcal{A}_{1}.$
Notice that, in the above inequality \eqref{Bdiss}, the constant $\beta >0$ can be chosen as 
$$\beta=\underline{\Sigma}+\underline{\sigma_{1}}+\varepsilon$$
for some arbitrarily small $\varepsilon > 0$ where
$$\underline{\Sigma}=\inf_{v \in \R^{3}}\dfrac{\mathbf{\Sigma}(v)}{1+|v|} >0 \qquad \text{ and } \qquad \underline{\sigma_{1}}=\inf_{v \in \R^{3}}\dfrac{\sigma_{1}(v)}{1+|v|} >0.$$
Notice $\beta$ does not depend on $R$.

\section{Spectral analysis of the linearised operator and its
  associated semigroup}

We recall that for $\alpha \in (\alpha_0,1]$, ${F}_\alpha$ denotes the
unique steady state with unit mass, solution to \eqref{equi}. In order
to study the stability of $\fe$ for $\alpha$ close to $1$ we will
first prove that the following operator has a spectral gap in
$\mathcal{X}$:
\begin{equation}
  \label{eq:Lalpha-Falpha}
  \mathscr{L}_{\alpha} (h)
  := \mathcal{Q}_\alpha(h,\fe) + \mathcal{Q}_\alpha(\fe,h)
  + \L(h)
  \qquad (h \in \mathcal{Y}).
\end{equation}
Notice that thanks to Proposition \ref{propoAlo} the above expression
is well defined and belongs to $\mathcal{X}$ whenever
$h \in \mathcal{Y}$. It is fundamental here that the domain of
$\mathscr{L}_{\alpha}$ in $\mathcal{X}$ \emph{does not depend on}
$\alpha$, i.e.
\begin{equation*}
  \D(\mathscr{L}_{\alpha})
  =
  \D(\mathscr{L}_{1})
  = \mathcal{Y} \qquad \forall \alpha \in (\alpha_{0},1].
\end{equation*}
This is in major contrast with the situations investigated in
\cite{MiMo2,MiMo3} where the forcing term was a differential operator
for which the domain of the associated linearised operator involved
Sobolev norms.

Since $\mathscr{L}_\alpha$ is the linearisation of the nonlinear
operator $\Q_\alpha(f,f) + \L(f)$ near its steady state $\fe$, a study
of its spectrum will allow us to perform a perturbative study of the
evolution equation \eqref{BEev}. We deduce from the results of Section
\ref{sec:convFeM} the following technical result stating that
$\mathscr{L}_\alpha$ is close to $\mathscr{L}_1$ for $\alpha$ close to
1. It will play a crucial role in our analysis:

\begin{propo}
  \label{prop:LaL1}
  There exists an explicit function
  $\varpi \colon (\alpha_{0},1] \to \R^{+}$ such that
  $\lim_{\alpha \to 1^{+}}\varpi(\alpha)=0$ and
  \begin{equation*}
    \|\mathscr{L}_{\alpha}(h)-\mathscr{L}_{1}(h)\|_{\mathcal{X}}
    \leq \varpi(\alpha)\,\|h\|_{\mathcal{Y}}
    \qquad \forall h \in \mathcal{Y}.
  \end{equation*}
\end{propo}

\begin{proof}
  A fundamental observation is that the domain of
  $\mathscr{L}_{\alpha}$ is actually independent of $\alpha$,
  i.e. $\D(\mathscr{L}_{\alpha})=\mathcal{Y}$ for any
  $\alpha \in (\alpha_{0},1]$. Let $h \in \mathcal{Y}$ be fixed. We
  have
  \begin{multline*}
   \|\mathscr{L}_{\alpha}(h) - \mathscr{L}_1(h)\|
    _{ {\mathcal{X}}}
    =
    \|\mathcal{Q}_\alpha(\fe, h) + \mathcal{Q}_\alpha(h, \fe)
    - \mathcal{Q}_1(\mathcal{M}, h) - \mathcal{Q}_1(h, \mathcal{M})
    \|_{ {\mathcal{X}}}
    \\
    \leq
    \|\mathcal{Q}_\alpha(\fe - \mathcal{M}, h)
    + \mathcal{Q}_\alpha(h, \fe-\mathcal{M})
    \|_{ {\mathcal{X}}}
    \\
    +
    \|
    \mathcal{Q}_\alpha(\mathcal{M}, h) - \mathcal{Q}_1(\mathcal{M}, h)
    \|_{ {\mathcal{X}}}+
    \|
    \mathcal{Q}_\alpha(h, \mathcal{M}) - \mathcal{Q}_1(h, \mathcal{M})
    \|_{ {\mathcal{X}}}.
   \end{multline*}
   Thus, using Proposition \ref{propoAlo} and Theorem \ref{limit1},
   one sees that there exists some constant $C >0$ and some explicit
   function $\eta_0(\alpha)$ with $\lim_{\alpha \to 1}\eta_0(\alpha)=0$ such that
   \begin{equation*}
     \|\mathcal{Q}_\alpha(\fe - \mathcal{M}, h)
     + \mathcal{Q}_\alpha(h, \fe-\mathcal{M}) \|_{ {\mathcal{X}}}
     \leq C
     \|\fe - \mathcal{M}\|_{{\mathcal{Y}}} \|h\|_{ {\mathcal{Y}}} \leq
     \eta_0(\alpha) \|h\|_{\mathcal{Y}}.
   \end{equation*}
   In the same way, according to Proposition \ref{elasticMM},
   \begin{equation*}
     \|\mathcal{Q}_\alpha(\mathcal{M}, h) - \mathcal{Q}_1(\mathcal{M}, h)\|_{ {\mathcal{X}}}+ \|\mathcal{Q}_\alpha(h, \mathcal{M})
     - \mathcal{Q}_1(h, \mathcal{M}) \|_{ {\mathcal{X}}} \leq
     \eta_1(\alpha)\|h\|_{\mathcal{Y}}
   \end{equation*}
   some explicit function $\eta_1(\alpha)$ with
   $\lim_{\alpha \to 1}\eta_1(\alpha)=0$. These two estimates give the
   result with $\varpi(\cdot)=\eta_{0}(\cdot)+\eta_{1}(\cdot).$
 \end{proof}
        
\subsection{Spectral gap of the linearised operator}
\label{sec:spectral-gap}

We investigate here the spectral properties of $\mathscr{L}_{\alpha}$
in $\mathcal{X}.$ We begin by studying the kernel of
$\mathscr{L}_\alpha$:

\begin{propo}
  \label{lem:simple}
  There exists some explicit $\alpha_{1} \in (\alpha_{0},1]$ such
  that, given $\alpha \in (\alpha_{1},1]$, $0$ is a simple and
  isolated eigenvalue of $\mathscr{L}_{\alpha}$ and there exists
  $G_{\alpha} \in \mathcal{Y}$ with unit mass and such that
  $$\mathscr{N}(\mathscr{L}_{\alpha})=\mathrm{Span}(G_{\alpha}) \qquad \forall \alpha \in (\alpha_{1},1].$$
  We denote then by $\mathbb{P}_{\alpha}$ the spectral projection
  associated to the zero eigenvalue of $\mathscr{L}_{\alpha}$. Then,
  for any $f \in \mathcal{X}$, one has
  $\mathbb{P}_{\alpha}f=\varrho_{f}\,G_{\alpha}$ with
  $\varrho_{f}:=\int_{\R^{3}}f(v)\d v.$ In particular,
  $\mathrm{Range}(\mathbb{I}-\mathbb{P}_{\alpha})=\widehat{\mathcal{X}}$
  for any $\alpha \in (\alpha_{1},1]$.
\end{propo}

\begin{proof}
  For any $\alpha \in (\alpha_{0},1]$, set for simplicity
  $T_\alpha=\mathscr{L}_1-\mathscr{L}_\alpha$ with domain
  $\D(T_{\alpha})=\mathcal{Y}.$ From Proposition \ref{prop:LaL1},
  $$\|T_\alpha h\|_\mathcal{X} \leq \varpi(\alpha)\|h\|_\mathcal{Y}  \qquad \forall h \in \D(\mathscr{L}_{1})=\mathcal{Y}.$$
  Since $\|\cdot\|_{\mathcal{Y}}$ is equivalent to the graph norm of
  $\D(\mathscr{L}_{1})$, there exists $c >0$ such that
  $$\|h\|_{\mathcal{Y}} \leq c\left(\|h\|_{\mathcal{X}}+\|\mathscr{L}_{1}h\|_{\mathcal{X}}\right)\qquad \forall h \in \mathcal{Y}$$
  from which the above inequality reads
  $$\|T_{\alpha} h\|_{\mathcal{X}}\leq a\|h\|_{\mathcal{X}}+b\|\mathscr{L}_{1}h\|_{\mathcal{X}} \qquad \forall h \in \D(\mathscr{L}_{1})$$
  with $b=c\,\varpi(\alpha)$. Since
  $\lim_{\alpha \to 1^{+}}\varpi(\alpha)=0$, this makes $T_\alpha$ a
  $\mathscr{L}_1$-bounded operator with relative bound $b <1$ for any
  $\alpha \in (\alpha_{0}',1]$ for some explicit
  $\alpha_{0}' \in (\alpha_0,1].$ In particular, according to
  \cite[Theorem 2.14, p. 203]{kato} (cf. Theorem
  \ref{theo:kato-delta}), the gap
  $\widehat{\delta}(\mathscr{L}_{\alpha},\mathscr{L}_{1})$ between
  $\mathscr{L}_\alpha=\mathscr{L}_1+T_\alpha$ and $\mathscr{L}_1$ (as
  defined in \cite[IV.2.4, p. 201]{kato}; see Appendix
  \ref{sec:perturbation}) is less than
  $\frac{\sqrt{2 b^{2}}}{1-b} =
  \frac{\sqrt{2}c\varpi(\alpha)}{1-c\varpi(\alpha)}$.
  Now, recall that the spectrum of $\mathscr{L}_{1}$ splits as
$$\mathfrak{S}(\mathscr{L}_{1})=\{0\} \cup \mathfrak{S}'(\mathscr{L}_{1})$$
where
$\mathfrak{S}'(\mathscr{L}_{1}) \subset \{z \in \mathbb{C}\,;\,\Re z
\leq -\nu\}$.
Denoting by $\mathbb{P}_{1}$ the spectral projection associated to the
$0$ eigenvalue, one gets that
$\mathcal{X}=\mathcal{X}^{''} \oplus \mathcal{X}^{'}$ with
$\mathcal{X}^{''}=\mathrm{Range}(\mathbb{P}_{1})$ and
$\mathcal{X}^{'}=\mathrm{Range}(\mathbb{I-P}_{1})$ with moreover
$\mathfrak{S}({\mathscr{L}_{1}}\vert_{ \mathcal{X}^{''}})=\{0\}$ and
$\mathfrak{S}({\mathscr{L}_{1}}\vert_{\mathcal{X}^{'}})=\mathfrak{S}'(\mathscr{L}_{1}).$
In particular, the two above parts of $\mathfrak{S}(\mathscr{L}_{1})$
are separated by the closed curve
$\gamma_{r}=\{z \in \mathbb{C}\,;\,|z|=r\}$, for any $r \in (0,\nu)$.
Then, according to \cite[Theorem 3.16, p. 212 \& IV.3.5]{kato} (see
Theorem \ref{theo:kato}), there exists $\delta > 0$ such that the same
separation of the spectrum and decomposition of $\mathcal{X}$ hold for
any operator $S \in \mathscr{C}(\mathcal{X})$ for which the gap
$\widehat{\delta}(S,\mathscr{L}_{1}) < \delta.$ Choosing now
$\alpha_{0}'' \in (\alpha_{0}',1]$ such that
$\widehat{\delta}(\mathscr{L}_{\alpha},\mathscr{L}_{1}) < \delta $ as
soon as $\alpha \in (\alpha_{0}'',1]$, one gets therefore that the
spectrum of $\mathfrak{S}(\mathscr{L}_{\alpha})$ can be separated by
$\gamma_{r}$, i.e. it splits as
$$\mathfrak{S}(\mathscr{L}_{\alpha})=\mathfrak{S}''(\mathscr{L}_{\alpha}) \cup \mathfrak{S}'(\mathscr{L}_{\alpha}) \qquad \forall \alpha \in (\alpha_{0}'',1]$$
where
$\mathfrak{S}''(\mathscr{L}_{\alpha}) \subset \{z \in
\mathbb{C}\,;\,|z| < r\}$
while
$\mathfrak{S}'(\mathscr{L}_{\alpha}) \subset \{z \in \mathbb{C}\,;\,|
z| > r\}$.
Moreover, the space $\mathcal{X}$ splits as
$\mathcal{X}=\mathcal{X}_{\alpha}^{''}\oplus \mathcal{X}_{\alpha}^{'}$
with
$\mathfrak{S}({\mathscr{L}_{\alpha}}\vert_{
  \mathcal{X}_{\alpha}^{''}})=\mathfrak{S}^{''}(\mathscr{L}_{\alpha})$
and
$\mathfrak{S}({\mathscr{L}_{\alpha}}\vert_{\mathcal{X}_{\alpha}^{'}})=\mathfrak{S}'(\mathscr{L}_{\alpha}).$
Moreover, still using Theorem \ref{theo:kato},
$\mathrm{dim}(\mathcal{X}_{\alpha}^{''})=\mathrm{dim}(\mathcal{X}_{\alpha}^{''})=1$. This
shows that actually
$$\mathfrak{S}^{''}(\mathscr{L}_{\alpha})=\{\mu_{\alpha}\}$$
where $\mu_{\alpha}$ is a \emph{simple eigenvalue} of
$\mathscr{L}_{\alpha}$ with $|\mu_{\alpha}| < r$ for any
$\alpha \in (\alpha_{0}'',1]$. Let us show that actually
$\mu_{\alpha}=0$ (at least for sufficiently large $\alpha$)
\footnote{Notice that this cannot be deduced directly from the fact
  that $r > 0$ can be chosen arbitrarily small since the range of
  parameters $(\alpha_{0}'',1]$ for which the above splitting holds
  actually depends on $r$ through the parameter $\delta$ in
  \ref{theo:kato}.}. Define $\mathbb{P}_{\alpha}$ as the spectral
projection operator associated to $\mu_{\alpha}$, i.e.
$$\mathbb{P}_{\alpha}=\frac{1}{2 \pi i}\oint_{\gamma_{r}} R(\xi,\mathscr{L}_{\alpha}) \d \xi \qquad \forall \alpha \in (\alpha_{1},1].$$
According to \ref{theo:kato}, one also has
$\lim_{\alpha \to
  1}\left\|\mathbb{P}_{\alpha}-\mathbb{P}_{1}\right\|_{\mathscr{B}(\mathcal{X})}=0$
with an explicit rate, from which there exists some explicit
$\alpha_{1} \in (\alpha_{0}'',1]$ such that
\begin{equation}
  \label{eq:normal}
  \|\mathbb{P}_{\alpha}f -\mathbb{P}_{1}f\|_{\mathcal{X}} < 1
  \qquad \text{ for all } \alpha \in (\alpha_{1},1], \, f \in \mathcal{X}.
\end{equation}
Let us prove that $\mu_{\alpha}=0$ for any
$\alpha \in (\alpha_{1},1]$. Let us argue by contradiction and assume
there exists $\alpha \in (\alpha_{1},1]$ for which
$\mu_{\alpha}\neq 0$. Let $\phi_{\alpha}$ be some normalized
eigenfunction of $\mathscr{L}_{\alpha}$ associated to $\mu_{\alpha}$,
i.e. $\phi_{\alpha} \in \mathcal{Y} \setminus \{0\}$ satisfies
$\mathscr{L}_{\alpha}
\phi_{\alpha}=\mu_{\alpha}\,\phi_{\alpha}$.
Integrating over $\R^{3}$ we get that
$$\IR \phi_{\alpha}(v)\d v=0.$$
For any $f \in \mathcal{X}$, there exists $\beta=\beta(\alpha,f)$ such
that $\mathbb{P}_{\alpha}f=\beta \phi_{\alpha}$ while
$\mathbb{P}_{1}f=\varrho_{f}\M$. In particular, one sees that
\begin{equation*}
  \IR \mathbb{P}_{\alpha}f \d v=0
  \qquad \text{ while }
  \qquad \IR \mathbb{P}_{1} f \d v = \varrho_{f} \qquad \forall f \in \mathcal{X}.
\end{equation*}
This clearly contradicts \eqref{eq:normal}. Therefore, for any
$\alpha \in (\alpha_{1},1]$, $\mu_{\alpha}=0$ and the above reasoning
shows that any associated eigenfunction $\phi_{\alpha}$ is such that
$$\IR \phi_{\alpha}(v)\d v \neq 0.$$
Let then $G_{\alpha}$ be the unique eigenfunction of
$\mathscr{L}_{\alpha}$ associated to the $0$ eigenvalue with
$\IR G_{\alpha}(v)\d v=1.$ From \cite[Eq. (6.34), p. 180]{kato}, one
has
$\mathrm{Range}(\mathbb{I-P}_{\alpha}) \subset
\mathrm{Range}(\mathscr{L}_{\alpha})$
for any $\alpha \in (\alpha_{2},1].$ Since
$\IR \mathscr{L}_{\alpha}f \d v=0$ for any
$f \in \D(\mathscr{L}_{\alpha})$ we get that
$\mathrm{Range}(\mathbb{I-P}_{\alpha}) \subset \widehat{\mathcal{X}}$
for any $\alpha_{2} < \alpha \leq 1.$ Thus, given $f \in \mathcal{X}$,
since $f=\mathbb{P}_{\alpha}f + (\mathbb{I-P}_{\alpha})f$, it holds
$$\varrho_{f}:=\int_{\R^{3}}f \d v=\int_{\R^{3}}\mathbb{P}_{\alpha}f \d v.$$
Since moreover $\mathbb{P}_{\alpha}f=\beta_{f}G_{\alpha}$ for some
$\beta_{f} \in \R$ and $G_{\alpha}$ is normalised, we get that
$\mathbb{P}_{\alpha}f=\varrho_{f}\fe.$ In particular, if
$f \in \widehat{\mathcal{X}}$ then $\mathbb{P}_{\alpha}f=0$ and
$f \in \mathrm{Range}(\mathbb{I-P}_{\alpha})$ which achieves the proof
of the result.\end{proof}

From now on, $\nu > 0$ will denote the size of the spectral gap of
$\mathscr{L}_1$ (see Theorem \ref{spect}). Our main result in this subsection is the following:
\begin{theo}
  \label{thm:spectral-gap-alpha}
  Take $0 < \nu_* < \nu$. There is $\alpha_0 < \alpha_* < 1$, with
  $\alpha_*$ depending on $\nu_*$, such that for all
  $\alpha_* < \alpha \leq 1$, the linear operator
  $\mathscr{L}_{\alpha}$ has a spectral gap of size $\nu_*$. More
  precisely, the spectrum of $\mathscr{L}_{\alpha}$ splits as
  $\mathfrak{S}(\mathscr{L}_{\alpha})=\{0\} \cup
 \mathfrak{S}\left(\mathscr{L}_{\alpha}\vert_{(\mathbb{I-P}_{\alpha})\mathcal{X}}\right)
   $
  with
  \begin{equation}
    \label{Eq:gapp}
    \mathfrak{S}\left(\mathscr{L}_{\alpha}\vert_{\widehat{\mathcal{X}}}\right)
    \subset \{\lambda \in \mathbb{C}\,;\,\Re\lambda \leq -\nu_{*}\}
    \qquad \forall \alpha \in (\alpha_*, 1]
  \end{equation}
 where we recall that $\mathbb{P}_{\alpha}$ denotes the spectral projection associated to the zero eigenvalue of $\mathscr{L}_{\alpha}\,$ and $\mathscr{L}_{\alpha}\vert_{\widehat{\mathcal{X}}}$ denotes the part of $\mathscr{L}_{\alpha}$ on $\widehat{\mathcal{X}}=(\mathbb{I-P}_{\alpha})\mathcal{X}.$
\end{theo}
To prove this we will use the following result asserting that if an
operator has a spectral gap, and another operator is close to it in a
certain sense, then it must also have a spectral gap of a comparable
size:

\begin{lemme}
  \label{lem:spectral-gap-perturbation}
  Let $X$ be a Banach space and let $(L_{0},\D(L_{0}))$ be the
  generator of a $C_{0}$-semigroup. For any $\varepsilon \in (0,1)$
  let $(L_{\varepsilon},\D(L_{\varepsilon}))$ be a given closed
  unbounded operator with $\D(L_{0}) \subset \D(L_{\varepsilon})$ and
  \begin{equation}
    \label{eq:sg-perturbation-smallness}
    \lim_{\varepsilon \to
      0}\left\|(L_{\varepsilon}-L_{0})R(\lambda,L_{0})\right\|_{\mathscr{B}(X)}=0
    \qquad \forall \lambda \in \mathbb{C}
    \quad \text{ with $\Re$}\lambda > s(L_{0}).
  \end{equation}
  Then,
  $$\limsup_{\varepsilon \to 0}s(L_{\varepsilon}) \leq s(L_{0})$$
  with
  \begin{equation*}
    \lim_{\varepsilon \to 0} \left\|
      R(\lambda,L_{\varepsilon}) - R(\lambda,L_{0})
    \right\|_{\mathscr{B}(X)}=0
    \qquad
    \forall \lambda \in \mathbb{C}
    \quad \text{ with $\Re$}\lambda > s(L_{0}).
  \end{equation*}
\end{lemme}

\begin{proof}
  Let $\lambda \in \mathbb{C}$ be given with
  $\Re\lambda > s(L_{0})$ and let $\varepsilon_{0} > 0$ be such that
$$\left\|(L_{0}-L_{\varepsilon})R(\lambda,L_{0})\right\|_{\mathscr{B}(X)} < 1\qquad \qquad \forall \varepsilon \in (0,\varepsilon_{0}).$$ 
Setting
$$\mathscr{J}_{\varepsilon}=I-(L_{\varepsilon}-L_{0})R(\lambda,L_{0})=:I- {Z}_{\varepsilon}$$ one gets that $\mathscr{J}_{\varepsilon}$ is invertible for any $\varepsilon \in (0,\varepsilon_{0})$ with $\mathscr{J}_{\varepsilon}^{-1}=\sum_{n=0}^{\infty} {Z}_{\varepsilon}^{n}.$ Since moreover
$$\mathscr{J}_{\varepsilon}=\left(\lambda-L_{\varepsilon}\right)R(\lambda,L_{0})$$
one gets that $\lambda-L_{\varepsilon}$ is invertible for any $\varepsilon \in (0,{\varepsilon}_{0})$ with 
$$R(\lambda,L_{\varepsilon})=R(\lambda,L_{0})\mathscr{J}_{\varepsilon}^{-1} \qquad \forall \varepsilon \in (0,\varepsilon_{0}).$$ Therefore, $\Re\lambda \geq s(L_{\varepsilon})$ for any $\varepsilon \in (0,\varepsilon_{0})$ which proves the first part of the result. For the second part, one has simply,  for a given $\lambda \in \mathbb{C}$ with $\Re\lambda > s(L_{0})$,
$$\left\|R(\lambda,L_{\varepsilon})-R(\lambda,L_{0})\right\|_{\mathscr{B}(X)}=\left\|R(\lambda,L_{0})\,\left(\mathscr{J}_{\varepsilon}^{-1}-I\right)\right\|_{\mathscr{B}(X)}$$
and it suffices to prove that 
$$\lim_{\varepsilon \to 0} \left\|\mathscr{J}_{\varepsilon}^{-1}-I\right\|_{\mathscr{B}(X)}=0.$$
Since $\mathscr{J}_{\varepsilon}^{-1}=\sum_{n=0}^{\infty}Z_{\varepsilon}^{n}$, we get
$$\left\|\mathscr{J}_{\varepsilon}^{-1}-I\right\|_{\mathscr{B}(X)}\leq \sum_{n=1}^{\infty}\left\|Z_{\varepsilon}^{n}\right\|_{\mathscr{B}(X)}$$
and clearly, since $\lim_{\varepsilon \to 0}\left\|Z_{\varepsilon}\right\|_{\mathscr{B}(X)}=0$ we get the conclusion.
\end{proof}

\begin{nb}
  We notice that the above result can also be seen as a way of stating
  general abstract results for relatively bounded operators (see
  \cite[Theorem 3.17, p. 214]{kato}).\end{nb}

We are now ready to prove Theorem \ref{thm:spectral-gap-alpha}:

\begin{proof}[Proof of Theorem \ref{thm:spectral-gap-alpha}]
  We will apply Lemma \ref{lem:spectral-gap-perturbation} to the
  restriction of $\mathscr{L}_{\alpha}$ to $\widehat{\mathcal{X}}$ for
  $\alpha$ close to $1$. (We recall the reader that the spaces
  $\widehat{\mathcal{X}}$ and $\widehat{\mathcal{Y}}$ were defined in
  section \ref{sec:spectr-elast}.) Notice that, since
 $$\int_{\R^{3}}\mathscr{L}_{\alpha}f(v)\d v=0 \qquad \forall f \in \D(\mathscr{L}_{\alpha}) \qquad \alpha \in (0,1]$$
 one can define the restriction
 $\widehat{\mathscr{L}}_{\alpha} \colon
 \D(\widehat{\mathscr{L}}_{\alpha}) \subset \widehat{\mathcal{X}} \to
 \widehat{\mathcal{X}}$
 by
 $\D(\widehat{\mathscr{L}}_{\alpha})=\D(\mathscr{L}_{\alpha}) \cap
 \widehat{\mathcal{X}}=\widehat{\mathcal{Y}}$
 and $\widehat{\mathscr{L}}_{\alpha}f=\mathscr{L}_{\alpha}f$ for any
 $f \in \widehat{\mathcal{Y}}$ for any $ \alpha \in (0,1].$ According
 to Theorem \ref{spect}, $s(\widehat{\mathscr{L}}_{1})=-\nu < 0$.

Estimate
  \eqref{eq:sg-perturbation-smallness} in Lemma
  \ref{lem:spectral-gap-perturbation} for $\widehat{\mathscr{L}}_1$ and
  $\widehat{\mathscr{L}}_\alpha$ is exactly Proposition \ref{prop:LaL1} since
  \begin{equation*}
    \|\widehat{\mathscr{L}}_{\alpha}(h)
    - \widehat{\mathscr{L}}_{1}(h)\|_{\widehat{\mathcal{X}}}
    = \|\mathscr{L}_{\alpha}(h) - \mathscr{L}_1(h)\|_{ {\mathcal{X}}}
    \qquad \forall h \in \widehat{\mathcal{Y}}.
  \end{equation*}
  Since
  $R(\lambda, \widehat{\mathscr{L}}_1) \colon \widehat{\mathcal{X}}
  \to \widehat{\mathcal{Y}}$,
  the hypotheses of Lemma \ref{lem:spectral-gap-perturbation} are
  satisfied and therefore
  $s(\widehat{\mathscr{L}}_{\alpha}) \leq \nu_* < \nu =
  s(\widehat{\mathscr{L}}_{1})$
  for any $\alpha$ close enough to $1$. Now,
  since $\widehat{\mathcal{X}}=\mathrm{Range}(\mathbb{I-P}_{\alpha})$
  for any $\alpha \in (\alpha_{1},1]$, one has
  $\widehat{\mathscr{L}}_{\alpha}=\mathscr{L}_{\alpha}\vert_{(\mathbb{I-P}_{\alpha})\mathcal{X}}$
  for any $\alpha \in (\alpha_{1},1].$ This finishes the proof.
\end{proof}

\subsection{Decay of the associated semigroup} 

Now, one should translate the above spectral gap of the operator
$\mathscr{L}_{\alpha}$ into a decay of the associated semigroup. To do
so, we use a stable splitting of the generator $\mathscr{L}_{\alpha}$
into a dissipative part and a regularising part. This strategy is
inspired in the recent results \cite{GMM,MS}, but we give a proof
adapted to our situation, exploiting a well known stability property
of the essential spectrum under weakly compact perturbations. Our
splitting is in the spirit of the one described in Section
\ref{sec:spectr-elast}. Namely, set
\begin{equation*}
  \mathcal{T}_{\alpha}f = \Q_{\alpha}(f,\fe) + \Q_{\alpha}(\fe,f),
  \qquad f \in \mathcal{Y}
\end{equation*}

so that $\mathscr{L}_{\alpha}=\mathcal{T}_{\alpha}+\L.$ The positive
part of this operator is
$$\mathcal{T}^{+}_{\alpha}f=\Q^{+}_{\alpha}(f,\fe)+\Q^{+}_{\alpha}(\fe,f)$$
and $\mathcal{T}_\alpha$ is written as
\begin{equation*}
  \mathcal{T}_{\alpha}f(v) =
  \mathcal{T}^{+}_{\alpha}f(v)
  -\fe(v)\int_{\R^{3}}f(w)\,|v-w|\d w-\sigma_{\alpha}(v)f(v),
  \qquad v \in \R^{3},\:f \in \mathcal{Y},
\end{equation*}
where 
$$\sigma_{\alpha}(v)=\int_{\R^{3}}\fe(w)\,|v-w|\d w \geq \underline{\sigma_{\alpha}} \left(1+|v|\right) \qquad \forall v \in \R^{3}.$$
Inspired by the splitting in Section \ref{sec:spectr-elast}, let us
pick $R >0$ large enough so that \eqref{Bdiss} holds true and define,
for any $\alpha$,
\begin{equation}
  \label{eq:Balpha}
  \mathcal{B}_{\alpha}f(v) =
  \mathcal{T}^{+}_{\alpha}(\chi_{B_R^{c}}f)(v)
  + \L^{+}(\chi_{B_{R}^{c}}f)(v)
  -\left(\mathbf{\Sigma}(v) + \sigma_{\alpha}(v)\right)f(v)
\end{equation}
for $f \in \mathcal{Y}$ and $v \in \R^{3}$. Let us first see that
$\mathcal{B}_\alpha$ thus defined is dissipative:
\begin{lemme}\label{lem:Bdiss}
  Let $R >0$ and $\beta >0$ be given as in \eqref{Bdiss}. For any
  $0 < \beta_{\star} < \beta$, there exists
  $\alpha^{\dagger}=\alpha^{\dagger}(\beta_{\star}) \in
  (\alpha_{0},1)$ such that
  \begin{equation}
    \label{Bdissal}
    \IR \mathrm{sign}f(v) \mathcal{B}_{\alpha}f(v)m^{-1}(v)\d v
    \leq -\beta_{\star} \|f\|_{\mathcal{Y}}
    \qquad \forall f \in \mathcal{Y},
    \qquad \forall \alpha \in (\alpha^{\dagger},1).
  \end{equation}
\end{lemme}

\begin{proof}
  The proof is a direct consequence of \eqref{Bdiss} together with the
  fact that $\mathcal{T}^{+}_{\alpha}$ converges strongly to
  $\mathcal{T}^{+}_{1}$. More precisely, let us fix
  $f \in \mathcal{Y}$ and compute first
  $\|\mathcal{T}^{+}_{\alpha}f_{R}-\mathcal{T}^{+}_{1}f_{R}\|_{\mathcal{X}}$. One
  checks easily that
  \begin{multline*}
    \|\mathcal{T}^{+}_{\alpha}f_{R}-\mathcal{T}^{+}_{1}f_{R}\|_{\mathcal{X}}
    \leq \|\Q_{\alpha}^{+}(\fe-\M,f_{R})+\Q^{+}_{\alpha}(f_{R},\fe-\M)\|_{\mathcal{X}}
    \\
    + \|\Q_{\alpha}^{+}(\M,f_{R})-\Q^{+}_{1}(\M,f_{R})\|_{\mathcal{X}}
    +\|\Q_{\alpha}^{+}(f_{R},\M)-\Q_{1}^{+}(f_{R},\M)\|_{\mathcal{X}}
    \\
    \leq
    C\,\|\fe-\M\|_{\mathcal{Y}}\,\|f_{R}\|_{\mathcal{Y}}
    + 2 p(1-\alpha)\|\M\|_{W^{1,1}_{1}(m^{-1})}\,\|f_{R}\|_{\mathcal{Y}}
  \end{multline*}
  where we used both Proposition \ref{propoAlo} and
  \ref{elasticMM}. Consequently, there exists a nonnegative function
  $\delta_{1}(\alpha)$ with $\lim_{\alpha \to 1}\delta_{1}(\alpha)=0$
  such that
  \begin{equation}
    \label{L+fr}
    \|\mathcal{T}^{+}_{\alpha}f_{R}-\mathcal{T}^{+}_{1}f_{R}\|_{\mathcal{X}} \leq \delta_{1}(\alpha)\|f\|_{\mathcal{Y}} \qquad \forall \alpha \in (\alpha_{0},1) \qquad \forall R > 0.\end{equation}
  Set now 
  $$\mathscr{F}_{\alpha}(f)=
  \IR \mathrm{sign}f(v) \mathcal{B}_{\alpha}f(v)m^{-1}(v)\d v \qquad \forall \alpha \in (\alpha_{0},1].$$
  Using the fact that 
  $$\mathcal{B}_{\alpha}f(v)-\mathcal{B}_{1}f(v)=\mathcal{T}^{+}_{\alpha}f_{R}(v)-\mathcal{T}^{+}_{1}f_{R}(v)-\left(\sigma_{\alpha}(v)-\sigma_{1}(v)\right)f(v)$$
  we get readily that
  \begin{multline*}
    \mathscr{F}_{\alpha}(f) \leq \mathscr{F}_{1}(f)+ \|\mathcal{T}^{+}_{\alpha}f_{R}-\mathcal{T}^{+}_{1}f_{R}\|_{\mathcal{X}} \\
    -\int_{\R^{3}}\,\left(\sigma_{\alpha}(v)-\sigma_{1}(v)\right)\,|f(v)|\,m^{-1}(v)\d v\\
    \leq \mathscr{I}_{1}(f)+\delta_{1}(\alpha)\,\|f\|_{\mathcal{Y}} 
    +\int_{\R^{3}}\left|\sigma_{\alpha}(v)-\sigma_{1}(v)\right|\,|f(v)|\,m^{-1}(v)\d v
  \end{multline*}
  where we used \eqref{L+fr}. Finally, since
  $|v-w| \leq \langle v\rangle \,\langle w\rangle$
  $\forall v,w \in \R^{3}$, we have
  $$|\sigma_{\alpha}(v)-\sigma_{1}(v)| \leq \int_{\R^{3}}|v-w|\,\left|\fe(w)-\M(w)\right|\d w \leq \langle v\rangle \|\fe-\M\|_{L^{1}_{1}(\R^{3})} \leq \langle v\rangle \|\fe-\M\|_{\mathcal{Y}}$$
  and we deduce from Theorem \ref{limit1} that 
  $$\int_{\R^{3}}\left|\sigma_{\alpha}(v)-\sigma_{1}(v)\right|\,|f(v)|\,m^{-1}(v)\d v \leq \eta_{1}(\alpha)\,\int_{\R^{3}}\langle v\rangle |f(v)|m^{-1}(v)\d v=\eta_{1}(\alpha)\,\|f\|_{\mathcal{Y}}$$
  with $\lim_{\alpha \to 1}\eta_{1}(\alpha)=0.$ To summarize, there exists a function $\delta(\cdot)$ with $\lim_{\alpha\to 1}\delta(\alpha)=0$ such that
  $$\mathscr{F}_{\alpha}(f) \leq \mathscr{F}_{1}(f)+\delta(\alpha)\,\|f\|_{\mathcal{Y}} \qquad \forall f \in \mathcal{Y}$$
  which, from  \eqref{Bdiss}, becomes
  $$\mathscr{I}_{\alpha}(f) \leq \left(\delta(\alpha)-\beta\right)\,\|f\|_{\mathcal{Y}} \qquad \forall f \in \mathcal{Y}.$$
  This gives the first part of result since $\lim_{\alpha\to
    1}\delta(\alpha)=0.$
\end{proof}

\medskip

We set now
$\mathcal{A}_{\alpha}=\mathscr{L}_{\alpha}-\mathcal{B}_{\alpha}$,
$\alpha \in (\alpha^{\dagger},1)$, or in other words
\begin{equation*}
  \mathcal{A}_{\alpha}f(v)=\mathcal{T}^{+}_{\alpha}(\chi_{B_{R}}f)(v)+\L^{+}(\chi_{B_{R}}f)(v)
  - \fe(v)\int_{\R^{3}}f(w)\,|v-w|\d w
\end{equation*}
for $v \in \R^{3}$ and $f \in \mathcal{X}$. Then we have then the following:

\begin{propo}\label{prop:ABalpha} For any $a_{\star}  \in (0, a)$ and for any $\alpha \in (\alpha^{\dagger},1)$, one has 
\begin{enumerate}[i)]
\item $\mathcal{A}_{\alpha} \in \mathscr{B}(X)$;
\item $\mathcal{B}_{\alpha}\::\: \D(\mathcal{B}_{\alpha}) \subset \mathcal{X} \to \mathcal{X}$ with domain $\D(\mathcal{B}_{\alpha})=\mathcal{Y}$ is the generator of a $C_0$-semigroup $\left(\mathcal{U}_{\alpha}(t)\right)_{t \geq 0}$ of $\mathcal{X}$ with
\begin{equation}\label{eq:Ualphat}
\left\|\,\mathcal{U}_{\alpha}(t)f\right\|_{\mathcal{X}} \leq \exp(-\beta_{\star}t)\|f\|_{\mathcal{X}} \qquad \forall t \geq 0,\:\:\forall f \in \mathcal{X}.\end{equation}
\end{enumerate}
\end{propo}
\begin{proof} It is clear from Proposition \ref{propoAlo} that $\mathcal{A}_{\alpha}$ is a bounded operator in $\mathcal{X}$ since 
\begin{multline*}
\|\mathcal{A}_{\alpha}f\|_{\mathcal{X}} \leq C\|\fe\|_{\mathcal{Y}}\|\chi_{B_{R}}f\|_{\mathcal{Y}} + \|\fe\|_{\mathcal{Y}}\|f\|_{L^{1}_{1}(\R^{3})} \\
\leq C_{R}\|\fe\|_{\mathcal{Y}}\,\|f\|_{\mathcal{X}}\,+\|\fe\|_{\mathcal{Y}}\|f\|_{\mathcal{X}} \qquad \forall f \in \mathcal{X}\end{multline*}
for some positive constant $C_{R}$ depending on $R$.

Since $\mathscr{L}_{\alpha}$ (with domain $\mathcal{Y}$) is the
generator of a $C_{0}$-semigroup in $\mathcal{X}$ according to Theorem
\ref{theo:generation} and $\mathcal{A}_{\alpha}$ is bounded, it
follows from the classical bounded perturbation theorem that
$\mathcal{B}_{\alpha}$ (with domain $\mathcal{Y}$) is also the
generator of a $C_{0}$-semigroup
$\left(\mathcal{U}_{\alpha}(t)\right)_{t \geq 0}$ in
$\mathcal{X}$. Since $\mathcal{B}_{\alpha}+\beta_{\star}$ is
\emph{dissipative} according to \eqref{Bdissal}, \eqref{eq:Ualphat}
holds according to the Lumer-Phillips theorem.
\end{proof}

Actually, it is easy to check that $\mathcal{A}_{\alpha}$ has better
regularising properties:
\begin{lemme} For any $\alpha \in (\alpha_{0},1)$, $\mathcal{A}_{\alpha} \in \mathscr{B}(\mathcal{X},\mathcal{Y}).$ Moreover, there exists some Maxwellian distribution $M$ such that, 
for any $\alpha \in (\alpha_{0},1)$,
$$\mathcal{A}_{\alpha} \in \mathscr{B}(\mathcal{X},\mathcal{H}) \qquad \text{ where } \mathcal{H}=L^{2}(M^{-1/2}).$$
In particular, $\mathcal{A}_{\alpha}$ is a weakly compact operator in
$\mathcal{X}$.
\end{lemme}
\begin{proof} Recall  that  there exist two Maxwellian distributions $\underline{\M}$ and $\overline{\M}$
  (independent of $\alpha$) such that
  $$\underline{\M}(v) \leq \fe(v) \leq \overline{\M}(v)
    \qquad \forall v \in \R^3,
    \qquad \forall \alpha \in (\alpha_0,1).$$
In particular, there exists some Maxwellian distribution $M$ such that 
$$\sup_{\alpha \in (\alpha_{0},1)}  \|\fe\|_{L^{2}_{2}(M^{-1})}=\sup_{\alpha \in (\alpha_{0},1)}\left(\int_{\R^{3}}M^{-1}(v)\langle v \rangle^{2}\,|f(v)|^{2}\d v \right)^{1/2} =C_{M}< \infty.$$
Then, because $|v-w| \leq \langle v\rangle\,\langle w\rangle$ for any
$v,w \in \R^{3}$, one has first that, for all $f \in \mathcal{X}$,
\begin{multline*}
\int_{\R^{3}}\left|\fe(v) \int_{\R^{3}} f(w)|v-w| \d w \right|^{2}
M^{-1}(v)\d v
\\
\leq
\IR |\fe(v)|^{2}\langle v\rangle^{2} M^{1}(v)\d v \,\left(\IR |f(w)|\langle w\rangle \d w\right)^{2}
\\
\leq \|\fe\|_{L^{2}_{2}(M^{-1/2})}^{2}\,\|f\|_{L^{1}_{1}}^{2}
\leq \|\fe\|_{L^{2}_{2}(M^{-1})}^{2}\,\|f\|_{\mathcal{X}}^{2}.
\end{multline*}
Moreover, according to \cite[Proposition 11]{AlCaGa}, there exists
$C > 0$ such that
\begin{align*}
  &\left\|\Q^{+}(g,h)\right\|_{L^{2}(M^{-1})}
    \leq C
    \|g\,M^{-1/2}\|_{L^{1}(\R^{3})}\,\|hM^{-1/2}\|_{L^{2}_{1}(\R^{3})}
    \quad \text{ and }
  \\ 
  &\left\|\Q^{+}(h,g)\right\|_{L^{2}(M^{-1})}
    \leq C
    \|g\,M^{-1/2}\|_{L^{1}_{1}(\R^{3})}\,\|hM^{-1/2}\|_{L^{2}(\R^{3})}.
\end{align*}
Using this with $g=f\chi_{B_{R}}$ and $h=\fe$ we get that there exists
$C > 0$ (independent of $\alpha$) such that
\begin{equation*}
  \left\|\mathscr{L}^{+}_{\alpha}(f\chi_{B_{R}})\right\|_{L^{2}(M^{-1})}
  \leq C \left(\|f\chi_{B_{R}}\|_{L^{1}_{1}(M^{-1/2})}
    + \|f\chi_{B_{R}}\|_{L^{1}(M^{-1/2})}\right).
\end{equation*}

In particular, there exists $C=C_{R} > 0$ such that
$$\left\|\mathscr{L}^{+}_{\alpha}(f\chi_{\B_{R}})\right\|_{L^{2}(M^{-1})} \leq C_{R} \|f\|_{\mathcal{X}}.$$
In the same way, 
$$\left\|\L^{+}(f\chi_{B_{R}})\right\|_{L^{2}(M^{-1})}\leq C_{R}\|f\|_{\mathcal{X}} \qquad \forall f \in \mathcal{X}.$$
This proves that
$\mathcal{A}_{\alpha} \in \mathscr{B}(\mathcal{X},\mathcal{H}).$ Due
to the Dunford-Pettis theorem the embedding
$\mathcal{H} \hookrightarrow \mathcal{X}$ is weakly compact, which
proves the second part of the Lemma.
\end{proof}

We this in hands, one has the following result about the decay of the
semigroup $\left(\mathcal{S}_\alpha(t)\right)_{t \geq 0}$ in
$\mathcal{X}$ generated by $\mathscr{L}_{\alpha}$. Remember that $\nu$
is the spectral gap of the elastic linearised operator
$\mathscr{L}_{1}$.

\begin{theo}\label{theo:local}
  Take $0 < \nu_* < \nu$ and $0 < \beta_{\star} < a$ (where $\beta >0$ is such
  that \eqref{Bdiss} holds). With the notations of Theorem
  \ref{thm:spectral-gap-alpha} and Lemma \ref{lem:Bdiss}, let
  $\alpha_{1}=\max(\alpha_{*},\alpha^{\dagger})$. Then, for any
  $\alpha \in (\alpha_{1},1)$ and for any $\mu \in (\nu_{*},\nu)$,
  there exists $C=C(\mu,\alpha) > 0$ such that
  \begin{equation*}
    \left\|
      \mathcal{S}_{\alpha}(t)\left(\mathbb{I-P}_{\alpha}\right)
    \right\|_{\mathscr{B}(\mathcal{X})}
    \leq
    C e^{-\mu t} \qquad \forall t \geq 0
  \end{equation*}
  where $\mathbb{P}_{\alpha}$ is the projection operator over
  $\mathrm{Span}(\fe)$ in $\mathcal{X}$. In other words, for any
  $h_0 \in \mathcal{X}$ the solution $h = h(t,v)$ (in the sense of
  semigroups) of the equation
  \begin{equation*}
    \partial_t h = \mathscr{L}_\alpha(h)
  \end{equation*}
  satisfies
  \begin{equation*}
    \| h(t) - c G_\alpha\|_{\mathcal{X}} \leq C \|h_0\| e^{-\mu t}
    \quad
    \text{ for $t \geq 0$},
  \end{equation*}
  where $c := \int_{\R^3} h_0$.
\end{theo}

\begin{proof}
  Denote by $(\mathcal{U}_\alpha(t))_{t \geq 0}$ the semigroup in
  $\mathcal{X}$ generated by $\mathcal{B}_\alpha$. Since
  $\mathscr{L}_{\alpha}=\mathcal{A}_{\alpha}+\mathcal{B}_{\alpha}$, it
  is well known from Duhamel's formula that
  \begin{equation*}
    \mathcal{S}_{\alpha}(t)
    = \mathcal{U}_{\alpha}(t)
    + \int_{0}^{t}\mathcal{S}_{\alpha}(t-s)
    \mathcal{A}_{\alpha}\,\mathcal{U}_{\alpha}(s)\d s.
  \end{equation*}
  Since $\mathcal{A}_{\alpha}$ is weakly compact in $\mathcal{X}$, one gets that, for any $t \geq s \geq 0,$ the integrand $\mathcal{S}_{\alpha}(t-s)\mathcal{A}_{\alpha}\,\mathcal{U}_{\alpha}(s)$ is a weakly compact operator in $\mathcal{X}$. Using then the ``strong compactness property'' (see \cite[Theorem C.7]{engel} for general reference and  \cite{Mmk} for the extension to weakly compact operators in $L^{1}$-spaces) we get that
  \begin{equation*}
    \text{ $\mathcal{S}_{\alpha}(t) -\mathcal{U}_{\alpha}(t)$
      is a weakly compact operator in
      $\mathcal{X}$ for all $t \geq 0$.}
  \end{equation*}
  We recall that, by definition, the Schechter essential spectrum is
  stable under compact perturbations. However, it can be shown that
  in $L^{1}$-spaces it is actually stable under \emph{weakly compact}
  perturbations, see \cite[Theorem 3.2 \& Remark 3.3]{latrach}. Due to
  this property we have
  \begin{equation*}
    \mathfrak{S}_{\mathrm{ess}}(\mathcal{S}_{\alpha}(t))
    = \mathfrak{S}_{\mathrm{ess}}(\mathcal{U}_{\alpha}(t))
    \qquad \forall t \geq 0.
  \end{equation*}
  In particular, the two $C_{0}$-semigroups share the same
  \emph{essential type}, i.e.
  $\omega_{\mathrm{ess}}(\mathcal{S}_{\alpha})=\omega_{\mathrm{ess}}(\mathcal{U}_{\alpha}).$
  Since
  $\omega_{\mathrm{ess}}(\mathcal{U}_{\alpha}) \leq
  \omega_{0}(\mathcal{U}_{\alpha})$ we get
  $$\omega_{\mathrm{ess}}(\mathcal{S}_{\alpha}) \leq -\beta_{\star} < 0.$$
  Since
  $$\omega_{0}(\mathcal{S}_{\alpha})=\max\left(\omega_{\mathrm{ess}}(\mathcal{S}_{\alpha}),s(\mathscr{L}_{\alpha})\right)$$
  with $s(\mathscr{L}_{\alpha})=0$ we obtain that 
  $$\omega_{0}(\mathcal{S}_{\alpha})=0 > \omega_{\mathrm{ess}}(\mathcal{S}_{\alpha}).$$
  General theory of $C_{0}$-semigroups \cite[Theorem V.3.1, page
  329]{engel} ensures that, for any
  $\omega > \omega_{\mathrm{ess}}(\mathcal{S}_{\alpha})$, one has
  $$\mathfrak{S}(\mathscr{L}_{\alpha}) \cap \{\lambda \in \mathbb{C}\,;\,\Re\lambda \geq \omega\}=\{\lambda_{1},\ldots,\lambda_{\ell}\}$$
  with $\lambda_{i}$ eigenvalue of $\mathscr{L}_{\alpha}$ with finite algebraic multiplicities $k_{i}$ and $\Re\lambda_{1}\geq \ldots \geq \Re\lambda_{\ell}$ and there is $C_{\omega} >0$ such that
  $$\left\|\mathcal{S}_{\alpha}(t)(I-\Pi_{\omega})\right\| \leq C_{\omega}\,\exp(\omega\,t) \qquad \forall t \geq 0$$
  where  $\Pi_{\omega}$ is the spectral projection associated to the set $\{\lambda_{1},\ldots,\lambda_{\ell}\}$.
  In particular, choosing $\ell=2$, since $\lambda_{1}=0$ while $\Re \lambda_{2}$ is the spectral gap of $\mathscr{L}_{\alpha}$ (see  Theorem  \ref{thm:spectral-gap-alpha}), choosing then $0< \mu <  \nu_{*}$, 
  $$\mathfrak{S}(\mathscr{L}_{\alpha}) \cap \{\lambda \in \mathbb{C}\,;\,\Re\lambda \geq -\mu\}=\{\lambda_{1}\}=\{0\}$$
  we get the result  since $\Pi_{\mu}=\mathbb{P}_{\alpha}$ is the
  spectral projection on the simple eigenvalue $0$.
\end{proof}

\section{Local stability of the steady state}
\label{sec:stability}

Using the result of the previous section, we show here that when $\alpha$ is close to $1$ and the initial
condition is close to the steady state, solutions to Eq. \eqref{BEev}
converge to it exponentially fast.  Choosing $\alpha$ close enough to
$1$ and the initial condition close enough to equilibrium, the
exponential speed of convergence can be as close to $\nu$ as we
want. In all this section, we shall assume that the initial datum
$f_0$ is such that there exist $b >0, s \in (0,1)$  satisfying

\begin{equation}
  \label{Hypf0}
  \IR f_0(v) \exp( b|v|)\d v < \infty.
\end{equation}
Then, according to Theorem \ref{tailT}, there exists $0 < a < b$
(independent of $\alpha$) such that the solution $f(t)$ of
\eqref{BEev} satisfies
$$f(t) \in \mathcal{X} \qquad \forall t \geq 0$$
where we recall that
\begin{equation*}
  \label{def:m}
  m(v)=\exp\left(-a |v|\right), \quad (a >0), \qquad \text{ and } \quad \mathcal{X}=L^{1}(\R^{3},m^{-1}(v)\d v).
\end{equation*}
 Actually, assuming a bit more on the
initial datum, one can prove the following where we recall that
$\mathcal{Y}=L^1_1(m^{-1})$:

\begin{lemme}
  \label{lem:evolution-Y-X-estimate}
  Let $f(t,v)$ be the solution to \eqref{BEev} associated to a
  nonnegative initial condition $f_0$ satisfying \eqref{mass} and
  \eqref{Hypf0}. Then, for any $\varepsilon > 0$ there exists $C > 0$
  depending only on $\varepsilon$ and the moment \eqref{Hypf0} of
  $f_0$ such that
  \begin{equation}
    \label{eq:evolution-Y-X-estimate}
    \|f(t)\|_{\mathcal{Y}} \leq C \|f(t)\|_{\mathcal{X}}^{1-\varepsilon}
    \qquad  \forall t \geq 0 .
  \end{equation}
\end{lemme}

\begin{proof}
  By H\"older's inequality, taking a dual pair $p$, $q$, i.e. $1/p+1/q=1$
  \begin{multline*}
    \|f(t)\|_\mathcal{Y}
    =
    \int_{\R^3} f(t,v) \langle v \rangle m^{-1}(v) \d v
    \\
    \leq
    \left(
      \int_{\R^3} f(t,v) m^{-1}(v) \d v
    \right)^{1/p}
    \left(
      \int_{\R^3} f(t,v) \langle v \rangle^q m^{-1}(v) \d v
    \right)^{1/q}
    \\
    =
    \|f(t)\|^{1/p}_{\mathcal{X}} \left(
      \int_{\R^3} f(t,v) \langle v \rangle^q m^{-1}(v) \d v
    \right)^{1/q}.
  \end{multline*}
  The initial datum $f_0$ has an exponential tail of order $1$ and
  Theorem \ref{tailT} ensures that this property propagates with time;
  hence for any $p \geq 1$ there exists $C >0$ such that
  \begin{equation*}
    \sup_{t \geq 0}  \left(
      \int_{\R^3} f(t,v) \langle v \rangle^q m^{-1}(v) \d v
    \right)^{1/q} \leq C < \infty,
  \end{equation*}
  where the constant $C >0$ depends only on $q$ and $f_0$ (and not on
  $\alpha$). Hence, choosing $p$ such that $1/p > 1-\varepsilon$
  proves the lemma.
\end{proof}

We deduce from the above lemma the following stability result:

\begin{theo}
  \label{thm:exp-convergence}
  Take $0 < \nu_* < \nu$. There exist $0 < \alpha_* < 1$ and
  $\varepsilon > 0$ such that, for all $\alpha$ with
  $\alpha_* \leq \alpha \leq 1$ and all nonnegative
  $f_0 \in \mathcal{Y}$ satisfying \eqref{mass} and \eqref{Hypf0} and
  such that $\|f_0 - \fe\|_\mathcal{X} \leq \varepsilon$, the solution
  $f(t)$ of \eqref{BEev} with initial data $f_0$ satisfies
  \begin{equation}
    \label{eq:exp-convergence}
    \|f(t) - \fe\|_\mathcal{X} \leq C \exp(-\nu_* t) \|f_0-\fe\|_\mathcal{X}
    \qquad (t \geq 0)
  \end{equation}
  for some constant $C$ which depends only on $\alpha$ and the moment
  \eqref{Hypf0} of the initial condition $f_0$.
\end{theo}

\begin{proof}
Let $\alpha \in (\alpha_0,1]$ and let $f(t,v)$ be the solution to \eqref{BEev} associated to the initial datum $f_0$. Setting $h(t,v) := f(t,v) - \fe(v)$, we may rewrite \eqref{BEev} as
  \begin{equation*}
    \partial_t h = \mathscr{L}_{\alpha}(h)
    + \mathcal{Q}_\alpha(h,h),
  \end{equation*}
  and through Duhamel's formula, denoting by
  $\left(\mathcal{S}_\alpha(t)\right)_{t \geq 0}$ the semigroup generated by
  $\mathscr{L}_{\alpha}$,
  \begin{equation*}
    h(t) =
    \mathcal{S}_\alpha(t) [h(0)]
    + \int_0^t \mathcal{S}_\alpha({t-s})
    [\mathcal{Q}_\alpha(h(s),h(s))] \,\d s.
  \end{equation*}
    Take any $\mu$ such that $\nu_* < \mu < \nu$ and consider $\alpha_0 < \alpha_* <
  1$ such that for every $\alpha_* \leq \alpha \leq 1$, the operator
  $\mathscr{L}_{\alpha}$ has a spectral gap of size $\mu$
  (the existence of such $\alpha_*$ is warranted by  Theorem
  \ref{thm:spectral-gap-alpha}).  Then, for this range of $\alpha$, the
  semigroup $((\mathbb{I}-\mathbb{P}_{\alpha})\mathcal{S}_\alpha(t))_{t\geq 0}$ decays exponentially with speed
  $\mu$. Recalling that $\mathrm{Range}(\mathbb{I-P}_{\alpha})=\widehat{\mathcal{X}}$ and that $h(t) \in \widehat{\mathcal{X}}$ for any $t \geq 0$ we get that there exists  $C > 0$ which depends only on $\alpha$ such that,
  \begin{multline*}
    \|h(t)\|_{\widehat{\mathcal{X}}}
    \leq
    \big\|
      \mathcal{S}_\alpha(t) [h(0)]
    \big\|_{\widehat{\mathcal{X}}}
    +
    \int_0^t \big\|
    \mathcal{S}_\alpha({t-s})
    [\mathcal{Q}_\alpha(h(s),h(s))]
    \big\|_{\widehat{\mathcal{X}}} \,\d s
    \\
    \leq
    C \| h(0)\|_{\widehat{\mathcal{X}}} \,\exp(-\mu t)
    +
    C \int_0^t
    \| \mathcal{Q}_\alpha(h(s),h(s)) \|_{\widehat{\mathcal{X}}}\, \exp(-\mu(t-s))
    \,\d s
    \\
    \leq
    C \| h(0)\|_{\widehat{\mathcal{X}}} \, \exp(-\mu t)
    +
    C \int_0^t
    \| h(s) \|_{\widehat{\mathcal{Y}}}^2 \, \exp(-\mu(t-s))
    \,\d s
    \\
    \leq
    C \| h(0)\|_{\widehat{\mathcal{X}}} \, \exp(-\mu t)
    +
    C_2 \int_0^t
    \| h(s) \|_{\widehat{\mathcal{X}}}^{3/2} \, \exp(-\mu(t-s))
    \,\d s,
  \end{multline*}
  where we have used Lemma \ref{lem:evolution-Y-X-estimate} with
  $\varepsilon = 1/4$.

  From this point, a Gronwall argument is enough to show that for
  $\|h(0)\|_{\widehat{\mathcal{X}}}$ small enough, $h(t)$ converges
  exponentially fast to $0$, with a speed as close to $\mu$ as we
  want. Let us develop this argument more precisely. Take $\delta > 0$
  and $f_0$ such that
  $\|h(0)\|_{\widehat{\mathcal{X}}} \leq \delta/2$, and consider $t$
  in a time interval $[0,T_\delta]$ where
  $\|h(t)\|_{\widehat{\mathcal{X}}} \leq \delta$. Then,
  \begin{equation*}
    \|h(t)\|_{\widehat{\mathcal{X}}}
    \leq
    C \| h(0)\|_{\widehat{\mathcal{X}}} \,\exp(-\mu t)
    +
    \delta^{1/2} C_2 \int_0^t
    \| h(s) \|_{\widehat{\mathcal{X}}} \, \exp(-\mu(t-s))
    \,\d s
  \end{equation*}
  or, rewriting this for the quantity $\gamma(t) := \exp(\mu t)
  \|h(t)\|_{\widehat{\mathcal{X}}}$,
  \begin{equation*}
    \gamma(t)
    \leq
    C \gamma(0)
    +
    \delta^{1/2} C_2 \int_0^t \gamma(s) \,ds
    \qquad (t \in [0,T_\delta]).
  \end{equation*}
  Then, by Gronwall's Lemma,
  \begin{equation*}
    \gamma(t) \leq C \gamma(0)\, \exp(\delta^{1/2} C_2 t)
    \qquad (t \in [0,T_\delta]),
  \end{equation*}
  or equivalently,
  \begin{equation}
    \label{eq:expconv-final}
    \|h(t)\|_{\widehat{\mathcal{X}}}
    \leq C \|h(0)\|_{\widehat{\mathcal{X}}} \,\exp((\delta^{1/2} C_2 - \mu) t)
    \qquad (t \in [0,T_\delta]).
  \end{equation}
  Now, take $\delta$ small enough so that $\delta^{1/2} C_2 - \mu < -
  \nu_*$, and $\varepsilon < \delta/2$ small so that
  \begin{equation*}
    C \varepsilon \, \exp(-\nu_* t)
    \leq \delta.
  \end{equation*}
  Then, for $\|h(0)\|_{\widehat{\mathcal{X}}} < \varepsilon$, one can actually take
  $T_\delta = +\infty$ and \eqref{eq:expconv-final} finishes the
  proof.
\end{proof}

\section{Global stability}

\subsection{Evolution of the relative entropy}

We consider the evolution of the relative entropy of a solution
$f(t,v)$ to \eqref{be:force} with respect to the equilibrium
Maxwellian $\M$. Notice that $\M$ is \emph{not the steady solution}
associated to \eqref{be:force} and $f(t,v)$ is not expected to
converge towards $\M$. However, we shall take advantage of the
entropy-entropy production estimate satisfied by $\L$ to estimate
first the distance from $f(t)$ to $\M$ which, combined with Theorem
\ref{limit1}, yields a control of the distance between $f(t)$ and
$\fe.$ Using the entropy-entropy production estimate, Theorem
\ref{theo:bcl1}, obtained in \cite{BCL1}, one has the following
crucial estimate on the evolution of the relative entropy along
solutions to \eqref{be:force}
\begin{propo}\label{prop:entropy}
For any $\alpha \in (\alpha_{0},1]$ and any nonnegative initial datum $f_{0}$ with unit mass and $H(f_{0}|\M) < \infty,$ the solution $f(t)=f(t,v)$ to \eqref{be:force} satisfies: 
$$H(f(t)|\M) \leq \exp(-\lambda t)H(f_{0}|\M) + K(1-\alpha) \qquad \forall t \geq 0\;;\;\alpha \in (\alpha_{0},1]$$
for some positive constant $K > 0$ independent of $\alpha$ where $\lambda >0$ is the constant appearing in Theorem \ref{theo:bcl1}
\end{propo}
\begin{proof} Let $\alpha \in (\alpha_{0},1]$ be fixed.
Given a solution $f(t)=f(t,v)$ to \eqref{be:force}, set  then
$$H(t)=H(f(t)|\M)=\IR f(t,v)\log\left(\dfrac{f(t,v)}{\M(v)}\right)\d v.$$
Computing the time derivative of $H(t)$ we get
\begin{equation*}
  \begin{split}
    \dfrac{\d}{\d t}H(t)
    &=\IR \Q_{\alpha}(f,f)\,\log\left(\dfrac{f}{\M}\right)\d \v
    + \IR \L f \log \left(\dfrac{f}{\M}\right)\d v
    \\
    & = \IR \Q_{\alpha}(f,f)\log f \d v
    -\IR \Q_{\alpha}(f,f)\log \M\d v + \mathbf{D}(f).
\end{split}\end{equation*}
We recall \cite{GPV**} that
\begin{equation}
  \label{dissDHal}
  \IR \Q_\alpha(g,g)(v) \log g(v)\d v
  = -\D_{H,\alpha}(g)
  + \frac{1-\alpha^2}{2\alpha^2}\IRR g(v)g(w)|v-w|\d v \d w,
\end{equation}
where the entropy production functional $\D_{H,\alpha}(g)$ is defined,
for any nonnegative~$g$, by
\begin{multline*}\D_{H,\alpha}(g)=\frac{1}{8\pi} \int_{\R^3 \times \R^3\times  \mathbb{S}^2} |v-w| g(v)g(w)\\
\times\left(\dfrac{g(v')g(w')}{g(v)g(w)} - \log\dfrac{g(v')g(w')}{g(v)g(w)} -1\right)\d\sigma\d v \d w  \geq 0\end{multline*}
where the post-collisional velocities $(v',w')=(v_\alpha',w_\alpha')$ are defined in \eqref{co:transf}. 
Moreover, using the definition of $\M$ and the fact that $\Q_{\alpha}$ conserves mass, one has
\begin{equation}\label{QlogM}\begin{split}
\IR \Q_{\alpha}(f,f)\log \M\d v&=-\dfrac{1}{2\Theta}\IR \Q_{\alpha}(f,f)|v|^{2}\d v\\
&=+\dfrac{1-\alpha^{2}}{16\Theta} \IRR f(t,v)f(t,w)\,|v-w|^{3}\d v\d w
\end{split}\end{equation}
where we used \eqref{co:weak}-\eqref{coll:psi} noticing that $\mathcal{A}_{\alpha}\left[|\cdot|^{2}\right](v,w)=-\frac{1-\alpha^{2}}{4}\,|v-w|^{2}.$ Putting \eqref{dissDHal} and \eqref{QlogM} together we obtain

\begin{multline*}
\dfrac{\d}{\d t}H(t)=-\mathbf{D}(f(t)) -\D_{H,\alpha}(f(t)) +\frac{1-\alpha^{2}}{2\alpha^{2}} \IRR f(t,v)f(t,w)|v-w|\d v\d w\\
-\dfrac{1-\alpha^{2}}{16\Theta} \IRR f(t,v)f(t,w)\,|v-w|^{3}\d v\d w\\
\leq -\lambda H(t)  +\frac{1-\alpha}{2\alpha^{2}_{0}}\IRR f(t,v)f(t,w)|v-w|\d v \d w.
\end{multline*}
Using the uniform control of moments of $f$
one sees that there exists some positive constant $C > 0$ independent
of $\alpha$ such that
\begin{equation*}\label{est}
\dfrac{\d}{\d t}H(t) \leq -\lambda H(t) + C\,(1-\alpha) \qquad \forall t \geq 0 \;,\;\forall \alpha \in (\alpha_{0},1].
\end{equation*}
Integrating this inequality, we obtain
$$H(f(t)|\M) \leq \exp(-\lambda t)H(f_{0}|\M)+\frac{C\,(1-\alpha)}{\lambda}\left(1-\exp(-\lambda t)\right)$$
which gives the result.
\end{proof}

\subsection{Global stability result}

Using the Csisz\'ar-Kullback inequality we deduce from Proposition
\ref{prop:entropy} the following
\begin{theo}
  There exist some $T=T(\alpha_{0})$ and some function
  $\ell\::\:(\alpha_{0},1] \to \mathbb{R}^{+}$ with
  $\lim_{\alpha \to 1}\ell(\alpha)=0$ such that
$$\|f(t)-\fe\|_{\mathcal{X}} \leq \ell(\alpha) \qquad \forall t \geq T(\alpha_{0})\,;\;\alpha \in (\alpha_{0},1].$$
\end{theo}
We remark that both $T$ and $\ell$ in the above theorem can be given
explicitly.

\begin{proof}
  Using both the Csisz\'ar-Kullback inequality and Holder's inequality we get
  \begin{align*}
    \|f(t)-\M\|_{L^{1}(m^{-1})}
    &\leq \|f(t)-\M\|_{L^{1}}^{1/2}\,\|f(t)-\M\|_{L^{1}(m^{-2})}^{1/2}
    \\
    &\leq \sqrt{2}H(f(t)|\M)\,\|f(t)-\M\|_{L^{1}(m^{-2})}^{1/2}.
  \end{align*}
  Due to Theorem \ref{tailT} (and recalling the definition of $m$ by
  \eqref{def:m}), assuming $a < r/2$ gives
  $$\sup_{t \geq 0} \|f(t)-\M\|_{L^{1}(m^{-2})} < \infty.$$
  Using now Proposition \ref{prop:entropy} we get the existence of two
  positive constants $C_{1},\, C_{2} >0$ independent of
  $\alpha \in (\alpha_{0},1]$ such that
  $$\|f(t)-\M\|_{\mathcal{X}} \leq C_1 \exp(-\lambda t) + C_{2}(1-\alpha) \qquad \forall t \geq 0\,;\,\alpha \in (\alpha_{0},1].$$
  This, combined with Theorem \ref{limit1}, yields
  \begin{equation}\label{convft}
    \|f(t)-\fe\|_{\mathcal{X}}\leq C_{1}\exp(-\lambda t) + \eta_{1}(\alpha) \quad \forall t \geq 0\,;\;\alpha \in (\alpha_{0},1]\end{equation}
  where $\eta_{1}(\cdot)$ is a given explicit function with $\lim_{\alpha \to 1}\eta_{1}(\alpha)=0.$ We get readily the conclusion.\end{proof}

The previous results essentially contain the proof of our main result:
\begin{proof}[Proof of Theorem \ref{theo:main}]
  With the notation of Theorem \ref{thm:exp-convergence}, one can pick
  $\alpha_{1} \in (\alpha_{0},1)$ so that
  $\ell(\alpha) \leq \varepsilon$ for any $\alpha \in (\alpha_{1},1)$
  so that
  $$\|f(t)-\fe\|_{\mathcal{X}} \leq \varepsilon \qquad \forall t \geq T(\alpha_{1})\::\:\forall \alpha \in (\alpha_{1},1]$$
  and Theorem \ref{thm:exp-convergence} yields the global stability
  result.
\end{proof}

\subsection*{Acknowledgments} B. L.  acknowledges support of the
\emph{de Castro Statistics Initiative}, Collegio C. Alberto,
Moncalieri, Italy. J.~A.~C.~was supported by the Marie-Curie CIG grant
KineticCF and the Spanish project MTM2011-27739-C04-02.

\appendix

\section{Some results in perturbation theory of linear operators}
\label{sec:perturbation}

We gather here some results that are needed in Section
\ref{sec:spectral-gap} in order to study the spectral properties of
$\mathscr{L}_\alpha$ for $\alpha$ close to 1. We begin by defining the
\emph{gap} between two closed linear operators, following
\cite[IV.2.4, p. 201]{kato}:

\begin{defi}
  Let $X$, $Y$ be Banach spaces and $S$, $T$ closed linear operators
  from $X$ to $Y$. Let $G(S)$, $G(T)$ be their graphs, which are
  closed linear subspaces of $X \times Y$. We set
  \begin{equation*}
    \delta(S,T) := \delta(G(S), G(T)) := \sup_{\substack{u \in G(S)\\ \|u\|_{X \times Y} = 1}}
    \dist (u, G(T)),
  \end{equation*}
  and we define the \emph{gap} between $S$ and $T$ as its symmetrisation:
  \begin{equation*}
    \hat{\delta}(S,T) := \max\{ \delta(S, T), \delta(T,S) \}.
  \end{equation*}
\end{defi}

We include here Theorem 2.14 from page 203 of \cite{kato} for the
convenience of the reader:

\begin{theo}[{\cite[Thm.~2.14, p.~203]{kato}}]
  \label{theo:kato-delta}
  Let $X$, $Y$ be Banach spaces and $A$, $T$ be closed operators between $X$
  and $Y$ such that $A$ is $T$-bounded with relative bound less than
  one; that is,
  \begin{equation*}
    \|A u\|_Y \leq a \|u\|_X + b \|T u\|_Y,
    \qquad u \in \D(T).
  \end{equation*}
  for some $a \geq 0$, $0 \leq b < 1$. Then $T+A$ is a closed operator
  and
  \begin{equation*}
    \hat{\delta}(S,T) \leq \frac{\sqrt{a^2+b^2}}{1-b}.
  \end{equation*}
\end{theo}

Finally, the following theorem is the main perturbation result we use
in order to deduce the properties of the spectrum of
$\mathscr{L}_\alpha$:

\begin{theo}[{\cite[Thm.~3.16, p.~212]{kato}}]
  \label{theo:kato}
  Let $T$ be a closed linear operator on a Banach space $X$ and assume
  its spectrum $\mathfrak{S}(T)$ is separated into two parts by a
  closed curve $\Gamma$ in $\C$. Let $X = X_T' \oplus X_T''$ be the
  associated decomposition of $X$. Then there exists $\delta > 0$,
  depending on $T$ and $\Gamma$, such that any operator $S$ on $X$
  with
  \begin{equation*}
    \hat{\delta}(S,T) < \delta
  \end{equation*}
  satisfies the following properties:
  \begin{enumerate}
  \item The spectrum $\mathfrak{S}(S)$ is also separated into two
    parts by the curve $\Gamma$.
  \item In the associated decomposition $X = X'_S \oplus X''_S$, the
    spaces $X'_s$ and $X''_S$ are respectively isomorphic to $X'_T$
    and $X''_T$.
  \item The decomposition $X = X'_S \oplus X''_S$ is continuous in $S$
    in the sense that the projection $P_S$ of $X$ onto $X'_S$ along
    $X''_S$ tends to $P_T$ in norm as $\hat{\delta}(S,T) \to 0$.
  \end{enumerate}
  
\end{theo}

\section{Proof that $\mathscr{L}_\alpha$ generates an evolution
  semigroup}

Let
$$m(v)=\exp(-a|v|), \qquad a >0$$
be fixed and let
$$\mathcal{X}=L^1(m^{-1}(v)\d v), \qquad \mathcal{Y}=L^1(\langle v \rangle m^{-1}(v)\d v).$$
We wish here to investigate the compactness properties of the 'gain'
part of $\mathscr{L}_\alpha$ for $\alpha \in (\alpha_0,1]$, with the
final aim of showing that $\mathscr{L}_\alpha$ generates a semigroup
for all $0 < \alpha \leq 1$. Recall that
$$\mathscr{L}_\alpha h=\Q_\alpha(h,\mathbf{F}_\alpha)+ \Q_\alpha(\mathbf{F}_\alpha,h)+\L(h), \qquad h \in \mathcal{Y}.$$
and also that
$$\L(h)=\mathcal{K}h -\Sigma(\cdot)h$$
where
$$\mathcal{K}h(v)=\Q^+_1(h,\M)(v)=\IR k(v,w)h(w)\d w
\quad \text{ and }
\quad \Sigma(v)=\IR \M(w)|v-w|\d w$$
with
\begin{equation}
  \label{k}
  k(v,w)=C_0|v-w|^{-1}\exp
  \left\{-\beta_0 \left(
      |v-w|+\dfrac{|v-\underline{u}|^2-{|w-\underline{u}|}^2}{|v-w|}\right)^2\right\}
\end{equation}
with $\beta_0=\frac{1}{8\Theta}$ and $C_0 >0$ a positive constant
(depending only on $\Theta_0$). Notice moreover that
$\Sigma(v)=\IR k(w,v)\d w.$ In the same way, one can write
$$\Q_\alpha(h,\mathbf{F}_\alpha)=\mathcal{K}_\alpha^1 h -\sigma_\alpha(\cdot) h \qquad \text{ and } \qquad \Q_\alpha(\mathbf{F}_\alpha,h)=\mathcal{K}_\alpha^2 h - \mathcal{K}_\alpha^3 h$$
with
$$\mathcal{K}_\alpha^1 h=\Q_\alpha^+(h,\mathbf{F}_\alpha), \qquad
\mathcal{K}_\alpha^2(h)=\Q^+_\alpha(\mathbf{F}_\alpha,h)$$
while
$$\sigma_\alpha(v)=\IR \mathbf{F}_\alpha(w)|v-w|\d w\qquad \text{ and } \qquad\mathcal{K}_\alpha^3(h)(v)=\mathbf{F}_\alpha(v)\IR h(w)|v-w|\d w.$$
With this notation
$$\mathscr{L}_{\alpha}h=\mathcal{K}h+\mathcal{K}_{\alpha}^{1}h+\mathcal{K}^{2}_{\alpha}h-\left(\Sigma+\sigma_{\alpha}\right)h -\mathcal{K}_{\alpha}^{3}h \qquad \forall h \in \mathcal{Y}.$$
As for $\L$, the two operators $\mathcal{K}_\alpha^i$, $i=1,2$ are integral operators with explicit kernels.
Namely,
\begin{lemme} For any $h \in \mathcal{Y}$, one has
$$\mathcal{K}_\alpha^1 h(v)=\IR K_\alpha^1(v,w) h(w)\d w$$
where
\begin{equation}\label{K1}
K_\alpha^1(v,w)=\dfrac{C_\alpha}{|v-w|}\int_{V_2 \cdot
(w-\v)=0}\mathbf{F}_\alpha\left(\v+V_2+\dfrac{\alpha-1}{\alpha+1}(w-\v)\right)\d
V_2
\end{equation}
for some positive constant $C_\alpha >0$.\end{lemme}
\begin{proof} The proof follows standard computations performed for instance in \cite{ArLo} where $\mathbf{F}_\alpha$ was replaced by a given Maxwellian.  In particular, \eqref{K1} is derived in \cite[p. 524]{ArLo}.\end{proof}

Recalling that  there exist two Maxwellian distributions $\underline{\M}$ and $\overline{\M}$
  (independent of $\alpha$) such that
  $$\underline{\M}(v) \leq \fe(v) \leq \overline{\M}(v)
    \qquad \forall v \in \R^3,
    \qquad \forall \alpha \in (0,1).$$
In particular, this proves that, for any $h \geq 0$,
$$\mathcal{K}_\alpha^1 h \leq  \overline{\mathcal{K}}_\alpha^1\,h=\Q_\alpha^+(h,\overline{\M})$$
and
$$\mathcal{K}_\alpha^2 h \leq \overline{\mathcal{K}}_\alpha^2 \,h =\Q_\alpha^+(\overline{\M},h).$$
Again, $\overline{\mathcal{K}}_\alpha^1\,h$ is an integral kernel with explicit kernel, namely
$$\overline{\mathcal{K}}_\alpha^1\,h(v)=\IR \overline{K}_\alpha^1(v,w)h(w)\d w $$
with
\begin{equation}\label{overK1}\overline{K}_\alpha^1 (v,w)=\overline{C}_\alpha\,|v-w|^{-1}\exp
\left\{-\beta_1\left(
(1+\mu_\alpha)|v-w|+\dfrac{|v-u_1|^2-{|w-u_1|}^2}{|v-w|}\right)^2\right\}\end{equation}
where $\overline{C}_\alpha >0$ is a positive constant depending only on $\alpha$ and $\overline{\M}$ while
$$\mu_\alpha=2 \frac{1-\alpha}{1+\alpha} \geq 0, \qquad \beta_1=\frac{1}{8{\Theta_1}};$$ ${\Theta_1}$, ${u}_1$ being the kinetic energy and momentum of $\overline{\M}$ (see \cite{ArLo}). By a simple domination argument (namely, Dunford-Pettis criterion), if $\overline{K}_\alpha^i$ are weakly compact in $\mathcal{X}$ then so will be $K_\alpha^i$, $i=1,2$.
\begin{propo} Let $\alpha \in (\alpha_0,1)$ be fixed. Then,
$$\mathcal{K} \::\:\mathcal{Y} \to \mathcal{X}, \qquad \mathcal{K}_\alpha^1\::\: \mathcal{Y} \to \mathcal{X} \:$$
are positive, bounded,  weakly compact operators. Moreover, $\mathcal{K}_{\alpha}^{2} \in \mathcal{B}(\mathcal{X})$ while
$$\mathcal{K}_\alpha^3\::\:\mathcal{X} \to \mathcal{X}$$
is a bounded and 
weakly compact operator.
\end{propo}
\begin{proof} The fact that $\mathcal{K},\,\mathcal{K}_\alpha^1$  are bounded operator from $\mathcal{Y}$ to $\mathcal{X}$ comes from Prop. \ref{propoAlo}. We divide the proof of the compactness properties into several steps.\\

\noindent {\it First step: weak compactness of $\mathcal{K}_\alpha^1$.} We already notice that it is enough to prove that $\overline{\mathcal{K}}_\alpha^1\::\:\mathcal{Y} \to \mathcal{X}$ is weakly compact. Let $\d\nu(v)=m^{-1}\d v$ and let $\mathcal{B}=B_\mathcal{Y}$ be the unit ball of $\mathcal{Y}$. Since $\mathcal{X}=L^1(\R^{3},\d\nu)$, according to Dunford-Pettis Theorem, this amounts to prove that
\begin{equation}\label{DP1}
\sup_{h \in \mathcal{B}}\int_{A} \left|\overline{\mathcal{K}}_\alpha^1 h(v)\right|\d\nu(v) \longrightarrow 0 \quad \text{ as } \quad  \nu(A) \to 0
\end{equation}
and
\begin{equation}\label{DP2}
\sup_{h \in \mathcal{B}}\int_{|v-u_1| > r} \left|\overline{\mathcal{K}}_\alpha^1 h(v)\right|\d\nu(v) \longrightarrow 0 \quad \text{ as } \quad r \to \infty.
\end{equation}
Using the representation of $\overline{\mathcal{K}}_\alpha^1$ as an integral operator, it is easy to check that \eqref{DP1} and \eqref{DP2} will follow if one is able prove that
\begin{equation}\label{DP11}
\sup_{w \in \mathbb{R}^3}\dfrac{m(w)}{\langle w \rangle}\int_{A} \overline{K}_\alpha^1(v,w)m^{-1}(v)\d v \longrightarrow 0 \quad \text{ as } \quad \nu(A) \to 0
\end{equation}
and
\begin{equation}\label{DP22}
\sup_{w \in \mathbb{R}^3}\dfrac{m(w)}{\langle w \rangle}\int_{|v-u_1| > r} \overline{K}_\alpha^1(v,w)m^{-1}(v)\d v  \longrightarrow 0 \quad \text{ as } \quad r \to \infty.
\end{equation}
Let us prove \eqref{DP11}. Let $A \subset \mathbb{R}^3$ be a given Borel subset and let $w \in \mathbb{R}^3$ be fixed. Set $B_w=\{v \in \mathbb{R}^3\;,\;|v-w| < 1\}.$ Since $\overline{K}_\alpha^1(v,w) \leq \overline{C}_\alpha|v-w|^{-1}$ one has
\begin{multline*}
\int_{A} \overline{K}_\alpha^1(v,w)m^{-1}(v)\d v \leq \overline{C}_\alpha\int_A |v-w|^{-1}\d\nu(v)\\
=\overline{C}_\alpha\left(\int_{A \cap B_w} |v-w|^{-1}\d\nu(v)+\int_{A \cap B_w^c}|v-w|^{-1}\d\nu(v)\right).
\end{multline*}
Clearly
$$\int_{A \cap B_w^c}|v-w|^{-1}\d\nu(v) \leq \nu(A)$$
while, for any $p >1$, $1/q+1/p=1$, one has
\begin{multline*}
\int_{A \cap B_w} |v-w|^{-1}\d\nu(v) \leq \left(\int_{A \cap B_w} \d\nu(v)\right)^{1/q}\,\left(\int_{A \cap B_w} |v-w|^{-p}\exp(a|v|)\d v\right)^{1/p}\\
\leq \exp\left(\frac{a}{p}\left(|w|+1\right)\right)\nu(A)^{1/q}\left(\int_{B_w} |v-w|^{-p}\d v\right)^{1/p}
\end{multline*}
where we used that, $\exp(a|v|) \leq \exp(a|w|) \exp(a|v-w|)$ for any $v,w$. Choosing now $p >3$, one sees that
$$\left(\int_{B_w} |v-w|^{-p}\d v\right)^{1/p} < \infty$$
and is independent of $w$. Thus, there exists $C=C(\alpha,a,p)$ such that
$$\int_{A} \overline{K}_\alpha^1(v,w)m^{-1}(v)\d v \leq C\left(\exp\left(\tfrac{a}{p}|w|\right)\nu(A)^{1/q}+\nu(A)\right) \qquad \forall w \in \mathbb{R}^3.$$
Since $p >1$, this proves that \eqref{DP11} holds true. Let us now prove \eqref{DP22}. One first notice that
\begin{equation}\label{DP22r}
\sup_{|w-u_1| \leq r/2}\dfrac{m(w)}{\langle w \rangle}\int_{|v-u_1| > r} \overline{K}_\alpha^1(v,w)m^{-1}(v)\d v  \longrightarrow 0 \quad \text{ as } \quad r \to \infty.
\end{equation}
Indeed, one notices that
$$\overline{K}_\alpha^1(v,w) \leq \overline{C}_\alpha\,|v-w|^{-1} \exp\left(2\beta_1(1+\mu_\alpha)\left[|w-u_1|^2 -|v-u_1|^2\right]\right).$$
Therefore, if $|w-u_1| \leq r/2$ and $|v-u_1| > r$, one gets
$$\overline{K}_\alpha^1(v,w) \leq \dfrac{2\overline{C}_\alpha}{r} \exp\left(-\tfrac{3}{2}\beta_1(1+\mu_\alpha)|v-u_1|^2\right)$$
and \eqref{DP22r} follows easily since
$$\IR \exp\left(-\tfrac{3}{2}\beta_1(1+\mu_\alpha)|v-u_1|^2\right)\d\nu(v) < \infty.$$
Now, to prove \eqref{DP22}, it is enough to show that
\begin{equation}\label{DP22rr}
\sup_{|w-u_1| > r/2}\dfrac{m(w)}{\langle w \rangle}\int_{|v-u_1| > r} \overline{K}_\alpha^1(v,w)m^{-1}(v)\d v  \longrightarrow 0 \quad \text{ as } \quad r \to \infty.\end{equation}
Arguing as in \cite[Proposition A.1]{BCL} (with $s=1$), there exists $K=K_\alpha >0$ such that
$$\int_{\mathbb{R}^3} \overline{K}_\alpha^1(v,w)m^{-1}(v)\d v \leq K_\alpha m^{-1}(w).$$
Therefore, for any $r >0$,
$$\dfrac{m(w)}{\langle w \rangle}\int_{|v-u_1| > r} \overline{K}_\alpha^1(v,w)m^{-1}(v)\d v \leq   K_\alpha{\langle w \rangle}^{-1}$$
and \eqref{DP22rr} follows since  $ \sup_{|w-u_1| > r/2} {\langle w \rangle}^{-1} \to 0$ as $r \to \infty.$ This achieves to prove that $\overline{K}_\alpha^1\::\:\mathcal{Y} \to \mathcal{X}$ is weakly compact.\\

\noindent {\it Second step: weak compactness of $\mathcal{K}.$} Notice that $\mathcal{K}$ and $\overline{\mathcal{K}}_\alpha^1$ are two integral operators whose kernels, given respectively by \eqref{k} and \eqref{overK1}, are very similar. The above computations can then be reproduced \textit{mutatis mutandis} to get the  weak-compactness of $\mathcal{K}$.\\

\noindent {\it Third step: boundedness of $\mathcal{K}_\alpha^2$}. According to \cite[Theorem 12]{AlCaGa}, and since $m(v)=\exp(-a|v|^{s})$ with $s=1$, one has
$$\overline{\mathcal{K}}_\alpha^2=\Q_\alpha^+(\overline{\M},\cdot)\::\:\mathcal{X} \to \mathcal{X}$$
is bounded for any $0 < \alpha < 1$. 
Consequently, simple domination argument asserts that $\mathcal{K}_{\alpha}^{2} \in \mathcal{B}(\mathcal{X})$.\\

\noindent {\it Final step: weak compactness of $\mathcal{K}_\alpha^3$}. Recall that
$$\mathcal{K}_\alpha^3 h (v)=\mathbf{F}_\alpha(v)\IR h(w)|v-w|\d w.$$
Therefore,
$$|\mathcal{K}_\alpha^3 h (v)| \leq \langle v \rangle \mathbf{F}_\alpha(v) \IR |h(w)|\langle w \rangle \d w.$$
In particular, there exists $C >0$ such that
$$|\mathcal{K}_\alpha^3 h (v)| \leq \langle v \rangle \mathbf{F}_\alpha(v) \IR |h(w)|m^{-1}(w) \d w.$$
Since
$$\int_{\mathbb{R}^3} \mathbf{F}_\alpha(v)\langle v \rangle m^{-1}(v)\d v < \infty$$
this proves that $\mathcal{K}_\alpha^3\::\:\mathcal{X} \to \mathcal{X}$ is bounded and dominated by a one-rank operator. In particular, it is weakly compact.
\end{proof}

As a general consequence, one has the following 
\begin{theo}\label{theo:generation}
For any $\alpha \in (\alpha_{0},1)$, the unbounded operator $\mathscr{L}_{\alpha}$ is the generator of a $C_{0}$-semigroup $\left(\mathcal{S}_\alpha(t)\right)_{t \geq 0}$in $\mathcal{X}$.
\end{theo}
\begin{proof} One applies the recent version of Desch theorem for positive semigroups in $L^{1}$-spaces, see for instance \cite{MMK}. For any $\alpha \in (\alpha_{0},1)$, define  the multiplication operator:
$$A_\alpha\,h(v)=-(\sigma_\alpha(v) + \Sigma(v)) h(v), \qquad h \in \mathscr{D}(A_\alpha)=\mathcal{Y}$$
then, $A_{\alpha}$ is the generator of a positive $C_{0}$-semigroup $(U_\alpha(t))_{t \geq 0}$ in $\mathcal{X}$. Since $\mathcal{K}$ and $\mathcal{K}_{\alpha}^{1}$ are weakly compact, one has from \cite{MMK} that $B_{\alpha}=A_{\alpha}+\mathcal{K}+\mathcal{K}_{\alpha}^{1}$ is the generator of a positive $C_0$-semigroup $(V_\alpha(t))_{t \geq 0}$. Finally, since $\mathcal{K}_{\alpha}^{2}$ and $\mathcal{K}_{\alpha}^{3}$ are bounded operators, one gets that 
$$\mathscr{L}_{\alpha}=B_{\alpha}+\mathcal{K}_{\alpha}^{2}-\mathcal{K}_{\alpha}^{3}$$
is the generator of a $C_{0}$-semigroup in $\mathcal{X}$.
 \end{proof}

\end{document}